\documentclass{amsart}
\usepackage{amssymb} 
\usepackage{amsbsy}
\usepackage{amscd}
\usepackage{amsmath}
\usepackage{amsthm}
\usepackage{amsxtra}
\usepackage{latexsym}
\usepackage[mathscr]{eucal}
\usepackage{upgreek}
\usepackage{wasysym}
\usepackage[all,cmtip]{xy}

\usepackage{xcolor}
\definecolor{rouge}{rgb}{0.85,0.1,.4}
\definecolor{bleu}{rgb}{0.1,0.2,0.9}
\definecolor{violet}{rgb}{0.7,0,0.8}
\usepackage[colorlinks=true,linkcolor=bleu,urlcolor=violet,citecolor=rouge]{hyperref}
\newcommand{\Wak}[2]{\mathbb{W}_{#1}^{#2}}

\newcommand{\V}{\mathbb{V}}

\newcommand{\wh}{\widehat}

\newcommand{\bra}{{\langle}}
\newcommand{\ket}{{\rangle}}

\newcommand{\on}{\operatorname}
\newcommand{\+}{\mathop{\oplus}}
\renewcommand{\*}{{\otimes}}
\newcommand{\mc}{\mathcal}
\newcommand{\mf}{\mathfrak}
\newcommand{\mb}{\mathbb}
\newcommand{\g}{\mf{g}}
\newcommand{\h}{\mf{h}}

\newcommand{\affg}{\widehat{\mf{g}}}

\newcommand{\isomap}{{\;\stackrel{_\sim}{\to}\;}}
\newcommand{\Z}{\mathbb{Z}}
\newcommand{\C}{\mathbb{C}}

\newcommand{\Q}{\mathbb{Q}}
\newcommand{\W}{\mathscr{W}}
\newcommand{\ra}{\rightarrow}
\newcommand{\lam}{\lambda}

\def\leq{\leqslant}
\def\geq{\geqslant}

\DeclareMathOperator{\Hom}{Hom}

\newcommand{\nc}{\newcommand}
\nc{\la}{\lambda}
\nc{\wt}{\widetilde}
\nc{\sw}{{\mathfrak s}{\mathfrak l}}
\nc{\ghat}{\wh{\g}}
\nc{\hhat}{\wh{\h}}
\nc{\bi}{\bibitem}
\nc{\pa}{\partial}
\nc{\ppart}{(\!(t)\!)}
\nc{\pparl}{(\!(\la)\!)}
\nc{\zpart}{(\!(z)\!)}
\nc{\n}{{\mathfrak n}}
\nc{\ol}{\overline}
\nc{\bb}{{\mathfrak b}}
\nc{\su}{\wh\sw_2}
\nc{\can}{\on{can}}
\nc{\ntil}{\wt{\n}}
\nc{\pone}{{\mathbb P}^1}
\nc{\bs}{\backslash}
\nc{\al}{\alpha}
\nc{\gt}{{\mathfrak g}'}
\nc{\ds}{\displaystyle}
\nc{\Bun}{\on{Bun}}
\nc{\Gr}{\on{Gr}}

\def\neg{\negthinspace}
\def\b{{\mathfrak b}}

\nc{\ka}{\kappa}
\nc{\cka}{\check\ka}
\nc{\cmu}{{\check\mu}}
\def\lg{{}^L\neg\g}
\def\lh{{}^L\neg\h}

\def\LG{{}^L\neg G}
\nc{\lka}{{}^L\neg\ka}
\nc{\LP}{{}^L\neg P}
\nc{\LQ}{{}^L\neg Q}
\nc{\lkt}{{}^L\neg\wt\ka}
\nc{\OO}{\mathcal O}
\nc{\Loc}{\on{Loc}}
\nc{\sto}{\!\!\shortto\!\!}
\nc{\hsl}{\widehat{\mathfrak sl}}

\theoremstyle{theorem}
\newtheorem{Th}{Theorem}[section]

\newtheorem{Pro}[Th]{Proposition}

\newtheorem{Lem}[Th]{Lemma}

\newtheorem{Co}[Th]{Corollary}

\theoremstyle{remark}

\numberwithin{equation}{section}

\title{Quantum Langlands duality of representations of ${\mc
    W}$-algebras}
\author{Tomoyuki Arakawa and Edward Frenkel}
\address{Research Institute for Mathematical Sciences, Kyoto University,
Kyoto 606-8502, JAPAN}
\email{arakawa@kurims.kyoto-u.ac.jp}

\address{Department of Mathematics, University of California,
  Berkeley, CA 94720, USA}

\begin{document}

\begin{abstract}
  We prove duality isomorphisms of certain representations of ${\mc
    W}$-algebras which play an essential role in the quantum
  geometric Langlands Program and some related results.
\end{abstract}

\maketitle

\section{Introduction}

Let $G$ be a connected, simply-connected simple algebraic group over
$\C$, $\LG$ its Langlands dual group, $\g=\on{Lie}(G)$,
$\lg=\on{Lie}(\LG)$.  By a level $\kappa$ we will mean a choice of a
symmetric invariant bilinear form on $\g$. We will denote by
$\check{\kappa}$ the level for $\lg$ whose restriction to the Cartan
subalgebra $\lh \subset \lg$ is dual to the restriction of $\kappa$ to
its Cartan subalgebra $\h \subset \g$ under the canonical isomorphism
$\h^* \cong \lh$.

Let $X$ be a smooth projective curve over $\C$.  Denote by
$\on{Bun}_G$ the moduli stack of principal $G$-bundles on $X$, and by
$D_\kappa\on{-mod}(\on{Bun}G)$ the derived category of
$(\kappa+\kappa_c)$-twisted $D$-modules on $\on{Bun}_G$.  Here
$\kappa_c$ corresponds to the critical level of $G$ (and square root
of the canonical line bundle on $\Bun_G$); that is,
$\kappa_c=\kappa_\g/2$, where $\kappa_{\g}$ is the Killing form of
$\g$.

In what follows, we call $\kappa$ {\em irrational} if
$\kappa/\kappa_\g \in \C \bs \Q$, and {\em generic} if
$\kappa/\kappa_\g \in \C \bs S$ for some countable subset $S \subset
\C$.

The global quantum geometric Langlands correspondence
\cite{Sto,GaiQL,FG} states that for irrational $\ka$ there should be
an equivalence of derived categories\footnote{A similar, but more
  subtle, equivalence is expected for rational values of $\ka$ as
  well.}
 \begin{align}
\mathbb{L}_\ka: D_\ka\on{-mod}(\on{Bun}_G) \isomap
D_{-\check\ka}\on{-mod}(\on{Bun}_{\LG}).
\end{align} 
In recent works \cite{GaiQL,CreGai,FG,Gaitsgory:uo}, various
constructions of the equivalence $\mathbb{L}_\ka$ have been proposed
that use representations of the ${\mc W}$-algebras ${\mc W}^\ka(\g)$
and ${\mc W}^{\cka}(\lg)$ and the isomorphism \cite{FeiFre91,FF}
$$
{\mc W}^\ka(\g) \cong {\mc W}^{\cka}(\lg).
$$
In particular, in D. Gaitsgory's construction \cite{Gaitsgory:uo}
an essential role is played by the duality isomorphisms
\begin{equation}    \label{main}
T^\ka_{\lam,\cmu} \cong \check{T}^{\cka}_{\cmu,\lam}.
\end{equation}
Here $\lam$ (resp., $\cmu$) is a dominant integral weight of $\g$
(resp., $\lg$); $T^\ka_{\lam,\cmu}$ and $\check{T}^{\cka}_{\cmu,\lam}$
are certain representations of ${\mc W}^\ka(\g)$ and ${\mc
  W}^{\cka}(\lg)$ (see Section \ref{fam mod} for the
definition). These isomorphisms also appeared in \cite{CreGai,FG} in a
similar context.

The goal of this paper is to prove the isomorphisms \eqref{main} for
irrational $\ka$ (see Theorem \ref{Th:symmetry} below) and some
closely related results.

The paper is organized as follows. In Section \ref{statement}, we
introduce the functor $H^\bullet_{DS,\cmu}(?)$ of quantum
Drinfeld--Sokolov reduction twisted by $\cmu \in \check{P}_+$ and the
family of modules $T^\ka_{\lam,\cmu} = H^0_{DS,\cmu}({\mathbb
  V}_{\lam,\ka})$, where ${\mathbb V}_{\lam,\ka}$ is the Weyl module
over $\ghat$ of highest weight $\lam \in P_+$ and level
$\ka+\ka_c$. We then state our main results: $H^j_{DS,\cmu}({\mathbb
  V}_{\lam,\ka}) = 0$ if $j \neq 0$ for any $\ka \in \C$ (Theorem
\ref{Th:van}); the isomorphisms \eqref{main} for irrational $\ka$
(Theorem \ref{Th:symmetry}); and irreducibility of $T^\ka_{\lam,\cmu}$
for irrational $\ka$ (Theorem \ref{Th:irr}).

In Section \ref{pr} we prove Theorem \ref{Th:symmetry} using a
realization of $T^\ka_{\lam,\cmu}$ for irrational $\ka$ as the
intersection of the kernels of powers of screening operators. In
Section \ref{irrvan}, using the results of \cite{Ara04,Ara07}, we
prove that $T^\ka_{\lam,\cmu}$ is irreducible for irrational $\ka$ and
to identify its highest weight as a ${\mc W}$-algebra module. Using
this fact, we give another proof of Theorem \ref{Th:symmetry}. Then we
prove the vanishing Theorem \ref{Th:van} in Section \ref{vangen} and
compute the characters of $T^\ka_{\lam,\cmu}$ in Section \ref{char
  form}. In Section \ref{fail} we show that the statement of Theorem
\ref{Th:symmetry} with rational $\kappa$ is false already for
$\g=\mf{sl}_2$. Finally, in Section \ref{higher} we construct a
BGG-type resolution of the modules for $T^\ka_{\lam,\cmu}$ with
irrational $\ka$ and discuss the $\ka \to \infty$ limit of this
resolution, following \cite{FF93}.

\subsection*{Acknowledgements}

We thank Thomas Creutzig, Dennis Gaitsgory and Davide Gaiotto for
valuable discussions. We are grateful to Sam Raskin for his helpful
comments on the first version of this paper; in particular, he pointed
out that a conjecture we had made in the first version about the
rational levels was incorrect.

T.A. is partially supported by JSPS KAKENHI Grant Number 17H01086,
\linebreak 17K18724.

\section{Statement of the main results}    \label{statement}

Let $\g$ be a simple Lie algebra over $\C$ of rank $r$, $\{
e_i,h_i,f_i \}$ its standard generators, $\g=\mf{n}_-\+
\h\+\mf{n}_+$ its triangular decomposition. Let $\Delta$ be
the set of roots of $\g$, $\Delta_+ \subset \Delta$ the set of
positive roots, $\Pi$ the set of simple roots, $P$ the weight lattice,
$\check{P}$ the coweight lattice. In what follows, we will use the
notation $e_\alpha$ (resp., $f_\alpha$) for a non-zero element of
$\n_+$ (resp., $\n_-$) corresponding to a root $\alpha \in \Delta_+$
(resp., $-\alpha \in \Delta_-$).

Let $\affg_{\kappa}=\g\ppart\+\C \mathbf{1}$ be the affine Kac--Moody
Lie algebra associated with $\g$ and level $\kappa+\kappa_c$, defined
by the commutation relation
$$[xf,yg]=[x,y]fg+(\kappa+\kappa_c)(x,y)\on{Res}_{t=0}(gdf)\mathbf{1},$$
$[\mathbf{1},\affg]=0$.
Let 
$$
V^\kappa(\g)=U(\affg_\kappa)\*_{U(\g[[t]]\* \C \mathbf{1})}\C,
$$
where $\C$ is the regarded as a one-dimensional representation of 
$\g[[t]]\* \C \mathbf{1}$ on which $\g[[t]]$ acts trivially and
$\mathbf{1}$ acts as the identity.
$V^\kappa(\g)$ is naturally a vertex algebra, and is called the {\em
  universal affine vertex algebra}
associated to $\g$ at level $\kappa+\kappa_c$.
A $V^\kappa(\g)$-module is the same as a smooth $\affg_{\kappa}$-module.

Let $H_{DS}^\bullet(M)$ be the functor of quantum Drinfeld--Sokolov
reduction with coefficients in a $\affg$-module $M$ \cite{FF90,FF}
(see Chapter 15 of \cite{FreBen04} for a survey). By definition, we
have
\begin{align*}
H_{DS}^\bullet(M)=H^{\infty/2+\bullet}(\mf{n}_+\ppart, M\* \C_{\Psi}),
\end{align*}
where $H^{\infty/2+\bullet}(\mf{n}_+\ppart,?)$ denotes the functor of
Feigin's semi-infinite cohomology of $\mf{n}_+\ppart$ and $\C_{\Psi}$
is the one-dimensional representation of $\mf{n}_+\ppart$
corresponding to a non-degenerate character $\Psi:\mf{n}_+\ppart\ra
\C$. The latter is defined by the formula
\begin{equation}    \label{psi}
\Psi(x \; f(t)) = \psi(x) \cdot \on{Res}_{t=0} f(t) dt, \qquad x \in
\mf{n}_+, \quad f(t)
\in \C\ppart,
\end{equation}
where $\psi$ is a character of $\mf{n}_+$ given by the formulas
\begin{equation}    \label{psial}
\psi(e_\al) = \begin{cases} 1, & \on{if} \; \al \; \on{is} \;
    \on{simple} \\ 0, & \on{otherwise} \end{cases}
\end{equation}
Formula \eqref{psi} shows that
if we identify the dual space to $\mf{n}_+\ppart$ with $\mf{n}_+^*\ppart
dt$ using the non-degenerate pairing between the latter and
$\mf{n}_+\ppart$ defined by the formula
$$
\bra \varphi \; g(t) , x \; f \ket = \bra \varphi,x \ket \cdot
\on{Res}_{t=0} f(t)g(t) dt, \qquad \varphi \in \n_+^*, \quad x \in \n_+,
$$
then $\Psi$ corresponds to the element $\psi \, dt \in
\mf{n}_+^*\ppart dt$.

Let $\bigwedge^{\infty/2+\bullet}(\mf{n}_+)$ be the fermionic ghosts
vertex algebra associated with $\mf{n}_+$. As a vector space, it is an
irreducible module over the Clifford algebra $\on{Cl}_{\n_+}$
associated to the vector space
$$
\mf{n}_+\ppart \oplus \mf{n}_+^*\ppart dt
$$
with a non-degenerate bilinear form induced by the above pairing. The
algebra $\on{Cl}_{\n_+}$ is topologically generated by $\psi_{\al,n} =
e_\al t^n, \psi^*_{\al,n} = e^*_\al t^{n-1} dt, \al \in \Delta_+, n
\in \Z$ with the relations
$$
[\psi_{\al,n},\psi^*_{\beta,m}]_+ = \delta_{\al,\beta}\delta_{n,-m},
\qquad [\psi_{\al,n},\psi_{b,m}]_+ =
[\psi^*_{\al,n},\psi^*_{\beta,m}]_+ = 0.
$$
The module $\bigwedge^{\infty/2+\bullet}(\mf{n}_+)$ is generated by a
vector $|0\ket$ such that
\begin{equation}    \label{ann}
\psi_{\al,n} |0\ket = 0, \quad n\geq 0, \qquad \psi^*_{\al,m} |0\ket =
0, \quad m>0.
\end{equation}
We define a $\Z$-grading on $\bigwedge^{\infty/2+\bullet}(\mf{n}_+)$
by the formulas $\on{deg} |0\ket = 0, \on{deg} \psi_{\al,n} = -1,
\on{deg} \psi^*_{\al,n} = 1$.

The graded space $H_{DS}^\bullet(M)$ is the cohomology of the complex
$(C(M),d)$, where
$$
C(M)=M\otimes \bigwedge{}^{\infty/2+\bullet}(\mf{n}_+).
$$
with respect to the differential
\begin{equation}    \label{dst}
d=d_{\on{st}}+\wh\Psi,
\end{equation}
where $d_{\on{st}}$ is the standard differential computing
semi-infinite cohomology
$$
H^{\infty/2+\bullet}(\mf{n}_+\ppart, M)
$$
(see formula (15.1.5) of \cite{FreBen04}) and $\wh\Psi$ stands for the
contraction operator on $\bigwedge^{\infty/2+\bullet}(\mf{n}_+)$
corresponding to $\Psi$ viewed as an element of $\mf{n}_+^*\ppart
dt$. In other words,
$$
\wh\Psi = \sum_{i=1}^r \psi^*_{\al_i,1}.
$$

It is known that $H_{DS}^i(V^\kappa(\g)) = 0$ for $i \neq 0$ (see
Theorem 15.1.9 of \cite{FreBen04}). The vertex algebra
$H_{DS}^0(V^\kappa(\g))$ is called the (principal) {\em ${\mc
    W}$-algebra} associated with $\g$ at level $\kappa+\kappa_c$. We
denote it by $\W^\kappa(\g)$.

We have the {\em Feigin--Frenkel duality} isomorphism
\cite{FeiFre91,FF}
\begin{align}
\W^\kappa(\g)\cong \W^{\check{\kappa}}(\lg),
\label{eq:FF-duality}
\end{align}
where $\lg$ is the Langlands dual Lie algebra to $\g$ and
$\check{\kappa}$ is the invariant bilinear form on $\lg$ that is dual
to $\kappa$ (see the Introduction).

\subsection{Twist}    \label{twist}

For $\cmu \in \check{P}$, we define a character $\Psi_\cmu$ of
$\mf{n}_+\ppart$ by the formula
\begin{equation}    \label{Psimu}
\Psi_\cmu(e_\alpha f(t)) = \psi(e_{\alpha}) \cdot \on{Res}_{t=0} f(t)
t^{\bra \cmu,\alpha \ket} dt, \qquad f(t) \in \C\ppart.
\end{equation}

Given a $V^\kappa(\g)$-module $M$, we define a new differential on the
complex $C(M)$:
\begin{equation}    \label{dmu}
d_\cmu = d_{\on{st}} + \wh\Psi_\cmu
\end{equation}
where $d_{\on{st}}$ is the standard differential appearing in
\eqref{dst} and $\wh\Psi_\cmu$ is the contraction operator
corresponding to the character $\Psi_\cmu$, viewed as an element
of $\mf{n}_+^*\ppart dt$:
$$
\wh\Psi_\cmu = \sum_{i=1}^r \psi^*_{\al_i,\bra \cmu,\alpha_i \ket+1}.
$$

We then define the functor $H^\bullet_{DS,\cmu}(?)$ by the formula
\begin{align*}
H_{DS,\check{\mu}}^\bullet(M)=H^{\infty/2+\bullet}(\mf{n}_+\ppart, M\*
\C_{\Psi_{\check{\mu}}}),
\end{align*}
where $\C_{\Psi_{\check{\mu}}}$ is the one-dimensional representation
of $\mf{n}((t))$ corresponding to the character $\Psi_{\check{\mu}}$.
We note that the functor $H^\bullet_{DS,\cmu}(?)$ has been studied in
\cite{FG}, Sect. 18 (mostly, in the critical level case $\ka=0$).

We define the structure of a $\W^\kappa(\g)$-module on
$H_{DS,\check{\mu}}^i(M), i \in \Z$, as follows.

For every $\cmu\in \check{P}$, let $\sigma_{\cmu}$ be the following
``spectral flow'' automorphism of $\affg_{\kappa}$:
\begin{align*}
e_i t^n &\mapsto e_i t^{n-\cmu_i}, \\
f_i t^n &\mapsto f_i t^{n+\cmu_i}, \\
h_i t^n &\mapsto h_i t^n - (\kappa+\kappa_c)(e_i,f_i) \cmu_i
\delta_{n,0},
\end{align*}
where
$$
\cmu_i = \langle \cmu,\al_i \rangle.
$$
Note that if $\cmu \in \LP = \on{Hom}(\C^\times,H) \subset \check{P}$,
then $\sigma_\cmu = \on{Ad}_{-\cmu(t)}$, where $-\cmu(t) \in H\ppart
\subset G\ppart$. For general $\cmu \in \check{P}$, $\sigma_\cmu \in
\on{Aut}(\affg_\ka)$ is a Tits lifting of the element of the extended
affine Weyl group corresponding to $\cmu$.

Given a $V^\kappa(\g)$-module (equivalently, a
$\affg_{\kappa}$-module) $M$, let $\sigma^*_\cmu M$ be the vector space
$M$ with the action of $\affg_{\kappa}$ twisted by the
automorphism $\sigma_{\cmu}$, i.e. $x \in \affg_{\kappa}$ acts as
$\sigma_\cmu(x)$. We will use the same notation
$\sigma^*_\cmu M$ for the corresponding $V^\kappa(\g)$-module.

We also define an automorphism similar to $\sigma_\cmu$ on the
Clifford algebra $\on{Cl}_{\n_+}$:
\begin{align*}
\psi_{\al,n} &\mapsto \psi_{\al,n-\cmu_i}, \\
\psi^*_{\al,n} &\mapsto \psi^*_{\al,n+\cmu_i}.
\end{align*}
Let $\sigma^*_\cmu \bigwedge{}^{\infty/2+\bullet}(\mf{n}_+)$ be the
twist of $\bigwedge{}^{\infty/2+\bullet}(\mf{n}_+)$, considered as a
$\on{Cl}_{\n_+}$-module, by this automorphism.

Combining these two automorphisms, we obtain an automorphism of
$C(V^\ka(\g)) = C(V^\ka(\g)) \otimes
\bigwedge{}^{\infty/2+\bullet}(\mf{n}_+)$ which we will also denote by
$\sigma_\cmu$. For any $V^\ka(\g)$-module $M$, let $\sigma^*_\cmu
C(V^\ka(\g))$ be the corresponding twist of $C(M) = M \otimes
\bigwedge{}^{\infty/2+\bullet}(\mf{n}_+)$ viewed as a module over the
tensor product of the enveloping algebra of $\ghat_\ka$ and
$\on{Cl}_{\n_+}$, or equivalently, as a module over the vertex algebra
$C(V^\ka(\g))$.

According to \cite{Li97}, the action of all fields from $C(V^\ka(\g))$
on $\sigma^*_\cmu C(V^\ka(\g))$ can be described explicitly by the
formula
\begin{equation}    \label{DeltaLi}
A \in C(V^\kappa(\g)) \mapsto Y_{C(M)}(\Delta(\cmu,z)A,z),
\end{equation}
where $\Delta(\cmu,z)$ is Li's delta operator (see \cite{Li97},
Section 3) corresponding to the field
\begin{equation}    \label{cmuz}
\cmu_i(z) + \sum_{\al \in \Delta_+} \bra \alpha_i,\cmu \ket :\! \psi_\al(z)
\psi^*_\al(z) \! :
\end{equation}
in $C(V^\kappa(\g))$, where $\cmu$ is viewed as an element of $\h =
\check{P} \underset{\Z}\otimes \C$.

The $\Z$-grading on $C(V^\ka(\g))$ and the differential $d$ given by
formula \eqref{dst} endow $(C(V^\ka(\g)),d)$ with the structure of a
differential graded vertex superalgebra. Its $0$th cohomology is
$\W^\kappa(\g)$ and all other cohomologies vanish. Furthermore,
$\W^\kappa(\g)$ can be embedded into the vertex subalgebra of
$C(V^\ka(\g))$ generated by the fields \eqref{cmuz} with $\cmu \in
\check{P}$ \cite{FF,FreBen04}. This vertex subalgebra is in fact
isomorphic to the Heisenberg vertex algebra $\pi^\ka$ and this
embedding is equivalent to the Miura map, see Section
\ref{pr} below for more details.

For any $\ghat_\ka$-module $M$, the complex $C(M) = M \otimes
\bigwedge{}^{\infty/2+\bullet}(\mf{n}_+)$ is naturally a
$C(V^\ka(\g))$-module. The $\W$-algebra $\W^\kappa(\g)$, viewed as a
subalgebra of $\pi^\ka \subset C(V^\ka(\g))$, acts on $C(M)$ and
therefore on the cohomology of $d$ on $C(M)$, which is
$H^\bullet_{DS}(M)$. This gives us a more explicit description of the
action of $\W^\kappa(\g)$ on $H^\bullet_{DS}(M)$.

Now take the $C(V^\ka(\g))$-module $\sigma^*_\cmu C(M)$. As a vector
space, it is $C(M)$, but it is equipped with a modified structure of
$C(V^\ka(\g))$-module; namely, the one obtained by twisting the action
by $\sigma_\cmu$ (see formula \eqref{DeltaLi}). Since $\pi^\ka$ is a
vertex subalgebra of $C(V^\ka(\g))$, we obtain that $\sigma^*_\cmu
C(M)$ is a $\pi^\ka$-module, and hence a $\W^\kappa(\g)$-module.
However, the action of $\W^\kappa(\g)$ now commutes not with $d$ but
with $\sigma_\cmu(d) = d_\cmu$. Indeed, we have
$$
\sigma_\cmu(d_{\on{st}}) = d_{\on{st}}, \qquad \sigma_\cmu(\wh\Psi) =
\wh\Psi_\cmu.
$$
Hence, under the $\sigma_\cmu$-twisted action, $\W^\kappa(\g)$
naturally acts on the cohomologies of the complex $C(M)$ with respect
to the differential $d_\cmu$. Thus, we obtain the structure of a
$\W^\kappa(\g)$-module on $H_{DS,\check{\mu}}^i(M), i \in \Z$.

\subsection{Family of modules}    \label{fam mod}

We define a family of modules over $\W^\kappa(\g)$ paramet\-rized
by $\lam \in P_+, \cmu \in \check{P}_+$.

For $\lam\in \h^*$, let $\mathbb{V}_{\lam}^\kappa$ denote the
irreducible highest weight representation of $\affg_\kappa$ with
highest weight $\lam$.

If $\lam \in P_+$, we also denote by $\mathbb{V}_{\lam,\kappa}$
the Weyl module induced from the irreducible finite-dimensional
representation $V_\lam$ of $\g$. If $\kappa$ is irrational and
$\lam \in P_+$, then $\mathbb{V}_{\lam}^\kappa =
\mathbb{V}_{\lam,\kappa}$.

In Section \ref{vangen} we will prove the following:

\begin{Th}    \label{Th:van}
For any $\kappa \in \C$ and any $\lam \in P_+, \cmu \in \check{P}$,
we have $H^j_{DS,\cmu}({\mb V}_{\lam,\ka}) = 0$ for all $j \neq 0$.
\end{Th}


It is easy to see that if $\cmu \in \check{P} \backslash \check{P}_+$,
then $H^0_{DS,\cmu}(\mathbb{V}_{\lam}^\kappa) = 0$ as well (see
Sect. 18 of \cite{FG}). However, if $\cmu \in \check{P}_+$, then the
$\W^\kappa(\g)$-module $H^0_{DS,\cmu}(\mathbb{V}_{\lam}^\kappa)$ is
non-zero.

Now we introduce our main objects of study in this paper, the modules
\begin{align}
T_{\lam,\cmu}^{\kappa} = H^0_{DS,\cmu}(\mathbb{V}_{\lam,\kappa}),
\qquad \lam \in P_+, \cmu \in \check{P}_+.
\end{align}
Theorem \ref{Th:van} implies a character formula for
$T_{\lam,\cmu}^{\kappa}$ which is independent of $\kappa$ (see Section
\ref{char form}). Because of that, the modules
$T_{\lam,\cmu}^{\kappa}$ may be viewed as specializations to different
values of $\ka$ of a single free $\C[\ka]$-module.


Switching from $\g$ to $\lg$, we also have the
$\W^{\check{\kappa}}(\lg)$-modules
$$
\check{T}_{\cmu,\lam}^{\check\kappa} =
H^0_{DS,\lam}(\mathbb{V}_{\cmu,\check\kappa}).
$$

The following theorem is the main result of this paper:

\begin{Th}\label{Th:symmetry}
Let $\kappa$ be irrational. Then for any $\lam \in P_+$ and $\cmu
\in \check{P}_+$ there is an isomorphism
\begin{equation}    \label{main1}
T_{\lam,\cmu}^{\kappa}\cong \check{T}_{\cmu,\lam}^{\check{\kappa}}
\end{equation}
of modules over $\W^\kappa(\g) \cong \W^{\check{\kappa}}(\lg)$.
\end{Th}

We will also prove the following result:

\begin{Th}    \label{Th:irr}
Let $\kappa$ be irrational. Then $T^\ka_{\lam,\cmu}$ is irreducible
for any $\lam \in P_+$ and $\cmu \in \check{P}_+$.
\end{Th}

A natural extension of the isomorphism \eqref{main1} with $\lam=0$ and
arbitrary $\cmu \in \check{P}_+$ to the case of the critical level
(i.e. $\ka=0$) has been proved in \cite{FG06}, Theorem 18.3.1(2), and
it is likely to hold for other $\lam \in P_+$ as well. For other
rational values of $\ka$, the isomorphism \eqref{main1} does not hold
for general $\lam$ and $\cmu$, even though the modules
$T_{\lam,\cmu}^{\kappa}$ and $\check{T}_{\cmu,\lam}^{\check{\kappa}}$
have equal characters for all $\ka$, according to the character
formula of Section \ref{char form}. The reason is that for rational
values of $\ka$ these two modules are usually reducible and have
different composition series.


\medskip

Let us comment on the role of Theorem \ref{Th:symmetry} in Gaitsgory's
work on the quantum geometric Langlands correspondence.

Let $\on{KL}(G)_\kappa$ be the category of $\affg_{\kappa}$-modules on
which $\g[[t]]$ acts locally finitely
and $t\g[t]$ acts locally nilpotently, and let $\on{Whit}(G)_{\kappa}$
be the category of $(\kappa+\kappa_c)$-twisted Whittaker $D$-modules on
the affine Grassmannian $\on{Gr}_G=G\ppart/G[[t]]$.  The fundamental
local equivalence that was conjectured by Gaitsgory and Lurie
and proved by Gaitsgory for irrational $\kappa$ states that there is an
equivalence $$FLE_{\kappa\ra \check{\kappa}}:\on{KL}(G)_\kappa\isomap
\on{Whit}(\LG)_{\check{\kappa}}$$ of chiral categories.  It follows
that there are two functors
\begin{align*}
& \on{KL}(G)_\kappa\otimes  \on{KL}(\LG)_{\check{\kappa}} \ra
\W^\kappa(\g)\on{-mod}
\end{align*}
given by
\begin{align*}
M\otimes N\mapsto H_{DS}^0(M\star FLE_{\check{\kappa}\ra \kappa}(N)),\quad M\otimes N\mapsto
H_{DS}^0(FLE_{\kappa\ra \check{\kappa}}(M)\star N),
\end{align*}
where $\star$ denotes the convolution product (see e.g.\ \cite{FG06}).
Theorem \ref{Th:symmetry} implies that these two functors coincide.
According to Gaitsgory \cite{Gaitsgory:uo}, a Ran space version of
this statement can be used to prove the existence of the quantum
geometric Langlands correspondence $\mathbb{L}_\ka$ discussed in the
Introduction. The isomorphism \eqref{main1} also appeared in a similar
context in \cite{CreGai,FG}.

\section{Proof of Theorem \ref{Th:symmetry}}    \label{pr}

Our proof uses a BGG-type resolution of the Weyl module
$\V_{\lam,\ka}$ with irrational $\ka$ in terms of the Wakimoto
modules. This resolution allows us to express $T^\ka_{\lam,\cmu}$ with
irrational $\ka$ as the intersection of the kernels of powers of the
screening operators acting on particular Fock representations of the
Heisenberg vertex algebra $\pi^\ka \cong \check\pi^{\check\ka}$. More
precisely, we obtain
$$
T^\ka_{\lam,\cmu} = \bigcap_{i=1}^r \on{Ker}_{\pi^\ka_{\lam-\ka\cmu}}
S^W_{i}(\lam_i+1), \qquad \check{T}^{\check\ka}_{\cmu,\la} =
\bigcap_{i=1}^r \on{Ker}_{\check\pi^{\check\ka}_{\cmu-\check\ka\lam}}
\check{S}^W_{i}(\cmu_i+1),
$$
where $r=\on{rank}\g$ and $S^W_{i}(\lam_i+1)$ and
$\check{S}^W_{i}(\cmu_i+1)$ are the operators introduced below. We
then show that
$$
\on{Ker}_{\pi^\ka_{\lam-\ka\cmu}} S^W_{i}(\lam_i+1) =
\on{Ker}_{\check\pi^{\check\ka}_{\cmu-\check\ka\lam}}
\check{S}^W_{i}(\cmu_i+1)
$$
for each $i=1,\dots,r$. The latter statements are independent from
each other for different $i$, and each of them reduces to the rank 1
case, i.e. the case of $\g=\sw_2$. In that case the kernels on both
sides are in fact known to be isomorphic to the same irreducible
representation of the Virasoro algebra \cite{Kac,FeFu,TsuKan86}. This
completes the proof of Theorem \ref{Th:symmetry}. The details are
given in the rest of this section.

In the next two sections we then present some further results.  In
Section \ref{sec:proof}, we use the results of \cite{Ara04,Ara07} to
prove that $T^\ka_{\lam,\cmu}$ is irreducible for all irrational $\ka$
and to identify its highest weight as a ${\mc W}$-algebra module. We
use this fact to give a different proof of Theorem \ref{Th:symmetry},
bypassing the information about representations of the Virasoro
algebra. Then we prove Theorem \ref{Th:van} and compute the characters
of $T^\ka_{\lam,\cmu}$. In Section \ref{higher}, we construct a
BGG-type resolution of the modules for $T^\ka_{\lam,\cmu}$ with
irrational $\ka$ and discuss the $\ka \to \infty$ limit of this
resolution, along the lines of \cite{FF93}.

\subsection{Heisenberg subalgebra}

Let $\kappa_0$ be the invariant bilinear form normalized so that the
square length of the maximal root of $\g$ is equal to $2$; that is,
$\ka_0 = \kappa_\g/2h^{\vee}$, where $h^{\vee}$ is the dual Coxeter
number of $\g$.  From now on, we will view $\kappa$ as a complex
number by identifying it with the ratio $\kappa/\kappa_0$.  Then the
complex numbers $\ka$ and $\check{\kappa}$ are related by the standard
formula:
$$\check{\kappa} = \frac{1}{m\kappa},$$
where $m$ is the lacing number of $\g$, i.e. the maximal number of the
edges in the Dynkin diagram of $\g$.

In what follows, we will use the notation $(\alpha|\beta)$ for
$\kappa_0(\alpha,\beta)$.

Let $\pi^\ka$ be the Heisenberg vertex algebra of level $\kappa$.
It is generated by
fields
$b_i(z)$, $i=1,\dots, r=\on{rank}\g$,
with the OPEs
\begin{align}
b_i(z)b_j(w)\sim \frac{\kappa(\alpha_i|\alpha_j)}{(z-w)^2}.
\label{eq:b_i}
\end{align}
Let $\pi^\ka_{\lam}$ be the irreducible highest weight representation
of $\pi^\ka$ with highest weight $\lam\in \h^*$.

Let $\Wak{\lam}{\ka}=M_{\g}\* \pi^\ka_{\lam}$ be the Wakimoto module
of highest weight $\lam$ and level $\kappa+\ka_c$
(\cite{FeuFre90,Fre05}), where $M_{\g}$ is the tensor product of $\dim
\mf{n}_+$ copies of the $\beta\gamma$ system.

The vacuum Wakimoto module $\Wak{0}{\ka}$ is naturally a vertex
algebra and there is an injective vertex algebra homomorphism
$V^\kappa(\g)\hookrightarrow \Wak{0}{\ka}$ \cite{Fre05}.

We can compute $H^\bullet_{DS}(\Wak{\nu}{\ka})$ by using a spectral
sequence in which the $0$th differential is $d_{\on{st}}$. It follows
from the construction of $\Wak{\nu}{\ka}$ that the $0$th cohomology of
$d_{\on{st}}$ is isomorphic to $\pi^\ka$ and all other cohomologies
vanish. Therefore the spectral sequences collapses and we obtain
$$
H_{DS}^0(\Wak{0}{\ka})\cong \pi^\ka.
$$
In fact, we can write down explicitly the fields in the complex
$(C(\Wak{\nu}{\ka}),d_{\on{st}})$ corresponding to the generating
fields $b_i(z)$ of $\pi^\ka$ \cite{FF} (the factor
$\frac{(\alpha_i|\alpha_i)}{2}$ in front of $h_i(z)$ is due to the
fact that $b_i(z)$ corresponds to the $i$th simple root rather than
coroot):
\begin{equation}    \label{bf b}
{\mathbf b}_i(z) = \sum_{n \in \Z} {\mathbf b}_{i,n} z^{-n-1} =
\frac{(\alpha_i|\alpha_i)}{2} h_i(z) +
\sum_{\al \in \Delta_+} (\alpha|\alpha_i):\! \psi_\al(z)
\psi^*_\al(z) \! :
\end{equation}
The first term contributes $\kappa+\kappa_c$ to the level, and the
second term contributes $-\ka_c$, so the total level is $\ka$, which
is indeed the level of $\pi^\ka$. Note that the field ${\mathbf
  b}_i(z)$ is nothing but the field $\cmu(z)$ given by formula
\eqref{cmuz} with $\cmu=\al_i$ (we identify $\h^*$ with $\h$ using the
inner product $\ka_0$). We have already mentioned the fact that these
fields generate the Heisenberg vertex algebra $\pi^\ka$ in Section
\ref{twist}.

By applying the functor $H_{DS}^0(?)$ to the embedding
$V^\kappa(\g)\hookrightarrow \Wak{0}{\ka}$, we obtain a vertex algebra
homomorphism \cite{FF90,FF}
\begin{align}
\Upsilon:\W^\kappa(\g)\ra H_{DS}^0(\Wak{0}{\ka})\cong \pi^\ka
\label{eq:Miura}
\end{align}
called the {\em Miura map}, which is injective for all $\kappa$ (see
e.g.\ \cite{AraLec}).  In particular, $\W^\kappa(\g)$ may be
identified with the image of the Miura map inside the Heisenberg
vertex algebra $\pi^\ka$. The latter can be described for generic
$\ka$ as the intersection of kernels of the screening operators
\cite{FF}. This fact can actually be taken as a definition of
$\W^\kappa(\g)$, see \cite{FF93}.

The $\mf{n}_+\ppart$-module $M_\g$ admits a right action $x \mapsto
x^R$ of $\mf{n}_+\ppart$ on $M_\g$ that commutes with the left action
of $\mf{n}_+\ppart$ \cite{Fre05}.  As a $U(\mf{n}_+\ppart)$-bimodule,
$M_\g$ is isomorphic to the semi-regular bimodule of $\mf{n}_+\ppart$
\cite{Vor99,A-BGG}, and hence we have the following assertion.

\begin{Pro}[{\cite[Proposition 2.1]{A-BGG}}]\label{Pro:key-iso}
Let $M$ be a $\mf{n}_+\ppart$-module 
that is integrable over $\mf{n}_+[[t]]$.
There is a linear isomorphism
\begin{align*}
\Phi:\Wak{\lam}{\ka}\* M\isomap \Wak{\lam}{\ka}\*  M
\end{align*}
such that
\begin{align*}
\Phi\circ \Delta(x)=(x\*1)\circ \Phi,
\quad \Phi \circ (x^R\*1)=(x^R\*1-1\* x) \circ \Phi\quad \text{for
}x\in \mf{n}_+\ppart.
\end{align*}
Here $x^R$ denotes the right action of $x\in \mf{n}_+\ppart$ on
$\Wak{\cmu}{\ka}$ and
$\Delta$ denotes the coproduct:
$\Delta(x)=x\*1 +1\*x$.\end{Pro}

Now we can describe the $\W^\ka(\g)$-modules
$H^i_{DS,\cmu}(\Wak{\nu}{\ka})$.

\begin{Lem}\label{pi shift}
For any $\nu \in \h^*$, we have
\begin{equation}    \label{HW}
H^i_{DS,\cmu}(\Wak{\nu}{\ka})\cong
  \delta_{i,0}\pi^\ka_{\nu-\ka\check{\mu}}
\end{equation}
as $\W^\ka(\g)$-modules.
\end{Lem}
\begin{proof}
By applying Proposition \ref{Pro:key-iso}
for $M=\C_{\Psi_{\cmu}}$,
we obtain a vector space isomorphism
\begin{align*}
H^{\infty/2+i}(\mf{n}_+\ppart,\Wak{\nu}{\ka}\*
\C_{\Psi_{\cmu}})\overset{\Phi}{\isomap}
H^{\infty/2+i}(\mf{n}_+\ppart, \Wak{\nu}{\ka})\*\C_{ \Psi_{\cmu}}\\
\cong H^{\infty/2+it}(\mf{n}_+\ppart, \Wak{\nu}{\ka})
\cong \delta_{i,0}\pi^\ka_{\nu}.
\end{align*}
According to the definition of the action of $\W^\ka(\g)$ on
$H^i_{DS,\cmu}(?)$ given in Section \ref{twist}, to obtain the
structure of a module over $\W^\ka(\g)$ we need to apply the
automorphism $\sigma_{\cmu}$ to the fields ${\mathbf b}_i(z)$ defined
by formula \eqref{bf b}. Under $\sigma_{\cmu}$, all ${\mathbf
  b}_{i,n}$ with $n \neq 0$ are invariant but ${\mathbf b}_{i,0}$ gets
shifted by $-\kappa \cmu_i$, where $\cmu_i = \bra
\cmu,\al_i^\vee\ket$. Indeed, $h_{i,0}$ gets shifted by
$-(\ka+\ka_c)\cmu_i$, and the $z^{-1}$-Fourier coefficient of the
fermionic term of \eqref{bf b} gets shifted by $\ka_c \cmu_i$. As the
result, we obtain that $H^0_{DS,\cmu}(\Wak{\nu}{\ka}) \cong
\pi^\ka_{\nu-\ka\check{\mu}}$.
\end{proof}

\subsection{Screening operators}

For each $i=1,\dots, r$,
the screening operator
$$
S_i(z): \Wak{\nu}{\ka}\ra \Wak{\nu-\alpha_i}{\ka}
$$
is defined in \cite{FF:wak,Fre05} by the formula
\begin{align}    \label{Si}
S_{i}(z)= e_i^R(z):e^{\int -\frac{1}{\kappa}b_i(z)dz}:,
\end{align}
where
\begin{align}
&:\! e^{\int -\frac{1}{\kappa}b_i(z)dz} \! : \; =
\label{eq:eb}
\\&T_{-\alpha_i} z^{-\frac{b_{i,0}}{\kappa}}
\exp\left(-\frac{1}{\kappa}\sum_{n<0}\frac{b_{i,n}}{n}z^{-n}\right)
\exp\left(-\frac{1}{\kappa}\sum_{n>0}\frac{b_{i,n}}{n}z^{-n}\right).
\nonumber
\end{align}
Here
$z^{-\frac{b_{i,0}}{\kappa}}=\exp(-\frac{b_{i,0}}{\kappa}\log z)$
and $T_{-\alpha_i}$ is the translation operator
$\pi^\ka_{\nu}\ra \pi^\ka_{\nu-\alpha_i}$
sending the highest weight vector to the highest weight vector and
commuting with all $b_{j,n}$, $n\ne 0$.

Let $\nu\in P$ be such that
\begin{align}
(\nu|\alpha_i)+m\kappa=\frac{(\alpha_i|\alpha_i)}{2}(n-1)
\label{eq:weight}
\end{align}
for some $n\in \Z_{\geq 0}$ and $m\in \Z$.
We have
\begin{align*}
&S_i(z_1)S_i(z_2)\dots S_i(z_{n})|_{\Wak{\nu}{\ka}}\\
&=\prod_{i=1}^{n}z_i^{-\frac{(\nu|\alpha_i)}{\kappa}}
\prod_{1\leq i<j\leq n}(z_i-z_j)^{\frac{(\alpha_i|\alpha_i)}{\kappa}}
:\! S_i(z_1)S_i(z_2)\dots S_i(z_n) \! :
\end{align*}
Let $\mc{L}^*_{n}(\nu,\kappa)$ be the local system with coefficients
in $\C$ associated to the monodromy group of the multi-valued function
$$\prod_{i=1}^{n}z_i^{-\frac{(\nu|\alpha_i)}{\kappa}}
\prod_{1\leq i<j\leq n}(z_i-z_j)^{\frac{(\alpha_i|\alpha_i)}{\kappa}}$$
on the manifold
$Y_{n}=\{(z_1,\dots, z_{n})\in (\C^*)^{n}\mid z_i\neq z_j\}$,
and denote by $\mc{L}_{n}(\nu,\kappa)$ the dual local system of
$\mc{L}^*(\nu,\kappa)$
(\cite{AomKit11}).
Then, for an element $\Gamma\in H_{n}(Y_{n},\mc{L}_n(\nu,\kappa))$,
\begin{align}
S_i(n,\Gamma):=\int_{\Gamma} S_i(z_1)S_i(z_2)\dots S_i(z_n) dz_1\dots
dz_n:\Wak{\nu}{\ka}\ra \Wak{\nu-n\alpha_i}{\ka}
\label{eq:intertwiner2}
\end{align}
defines a
$\affg$-module homomorphism.

\begin{Th}[\cite{TsuKan86}]\label{Th:Tsuhiya-Kanie}
Suppose that
$$\frac{2d(d+1)}{\kappa(\alpha_i|\alpha_i)}\not\in \Z,\quad
\frac{2d(d-n)}{\kappa(\alpha_i|\alpha_i)}\not\in \Z,
$$
for all $1\leq d\leq n-1$.
Then there exits a cycle
$\Gamma\in H_{n}(Y_{n},\mc{L}_n(\nu,\kappa))$ such that  
$S_i(n,\Gamma)$ 
is non-zero.
\end{Th}

In fact, it follows from more general results in \cite{SV,Varchenko}
(see \cite{FF93} for a survey) that for irrational $\ka$ the
cohomology group $H_{n}(Y_{n},\mc{L}_n(\nu,\kappa))$ is
one-dimensional. We will choose once and for all its generator
$\Gamma$ and will write $S_i(n)$ for the corresponding operator
$S_i(n,\Gamma)$.

The following theorem was proved for $\lam=0$ in \cite{FF}, and for
general $\lam \in P_+$ in \cite{ACL17}.

\begin{Pro}\label{Pro:resolution-weyl}
Let $\kappa$ be irrational and $\lam\in P_+$.
Then there exists a resolution $C^\bullet_\lam$
of the Weyl module $\mathbb{V}_\lam^\kappa = \mathbb{V}_{\lam,\kappa}$
of the form
\begin{align}    \label{resWak}
&0\ra \mathbb{V}_\lam^\kappa \ra C^0_\lam\overset{d^0_\lam}{\ra} C^1\ra\dots \ra
C^n\ra 0,\\ \notag
& C^j_\lam=\bigoplus_{w\in W\atop \ell(w)=j}\Wak{w\circ \lam}{\ka}, \qquad
w \circ \lam = w(\lam+\rho)-\rho,
\end{align}
with the differential $d^0_\lam$ given by
\begin{equation}    \label{d0}
d^0_\lam=\sum_{i=1}^{r} c_i S_i(\lam_i+1)
\end{equation}
for some $c_i \in \C$ with $\lam_i=\bra \lam,\alpha_i^\vee\ket$.
\end{Pro}

\begin{proof}
We recall the proof for completeness. Let $M^*_\nu$ be the
contragradient Verma module over $\g$ with highest weight $\nu \in
\h^*$. Let $M^{*\ka}_\nu$ be the corresponding induced $\ghat$-module
of level $\ka$. From the explicit construction of the Wakimoto module
$\Wak{\nu}{\ka}$ (see \cite{Fre05}) it follows that the degree
0 subspace of $\Wak{\nu}{\ka}$ (with respect to the Sugawara operator
$L_0$ shifted by a scalar so that the highest weight vector has degree
0), is isomorphic to $M^*_\nu$ as a $\g$-module. Therefore we have a
canonical homomorphism $M^{*\ka}_\nu \to \Wak{\nu}{\ka}$ which is an
isomorphism on the degree 0 subspaces.

If this homomorphism were not injective, then its kernel would contain
a singular vector of strictly positive degree. Consider then the
canonical homomorphism from $M^{*\ka}_\nu$ to the contragradient
module of the Verma $\ghat$-module $M^\ka_\nu$, which induces an
isomorphism of degree $0$ subspaces. The presence of such a singular
vector in $M^{*\ka}_\nu$ implies that the Verma module $M^\ka_\nu$
would also contain a singular vector of positive degree. However, if
$\nu \in P$ and $\ka$ is irrational, it is known that there are no
such singular vectors in $M^\ka_\nu$. Therefore we find that in this
case the homomorphism $M^{*\ka}_\nu \to \Wak{\nu}{\ka}$ is injective.
Since these two $\ghat$-modules have the same character, we obtain
that $M^{*\ka}_\nu \cong \Wak{\nu}{\ka}$ if $\nu \in P$ and $\ka$ is
irrational.

Now let $\lam \in P_+$. Then we have the contragradient BGG resolution
$C^\bullet_\la(\g)$ of the irreducible $\g$-module $V_\la$ with
highest weight $\lam$ such that
$$
C^j_\la(\g)=\bigoplus_{w\in W\atop \ell(w)=j} M^*_{w\circ \lam}.
$$
Let $C^j_\la(\ghat)$ be the induced resolution of $\ghat$-modules of
level $\ka$. Then for irrational $\ka$ we have
$$
C^j_\la(\ghat)=\bigoplus_{w\in W\atop \ell(w)=j} M^{*\ka}_{w\circ
  \lam} \simeq \bigoplus_{w\in W\atop \ell(w)=j} \Wak{w\circ
  \lam}{\ka}.
$$
In particular, the 0th differential $d^0_\la: C^0_\la(\ghat) \to
C^1_\la(\ghat)$ is the sum of non-zero homomorphisms $\phi_i: M^{*\ka}_\lam
\to M^{*\ka}_{\lam-(\lam_i+1)\al_i}$, or equivalently, $\Wak{\lam}{\ka} \to
\Wak{-(\lam_i+1)\al_i}{\ka}$. Since
$\Hom_{\ghat}(M^{*\ka}_\lam,M^{*\ka}_{\lam-(\lam_i+1)\al_i}) \cong
\Hom_{\g}(M^{*\ka}_\lam,M^{*\ka}_{\lam-(\lam_i+1)\al_i})$ is one-dimensional,
and $S_i(\lam_i+1)$ is a non-zero homomorphism $\Wak{\lam}{\ka} \to
\Wak{-(\lam_i+1)\al_i}{\ka}$ by Theorem \ref{Th:Tsuhiya-Kanie}, we
find that $d^0_\la$ is given by formula \eqref{d0}.
\end{proof}

%

The $\ghat$-homomorphism
$S_i(\lam_i+1):\Wak{\lam}{\ka}\ra \Wak{\lam-(\la_i+1)\alpha_i}{\ka}$
induces a linear map
\begin{align}
H_{DS,\cmu}^0(\Wak{\lam}{\ka})\ra
H_{DS,\cmu}^0(\Wak{\lam-(\lam_i+1)\alpha_i}{\ka})
\label{eq:induced-screaning}
\end{align}
for $\lam\in P_+$.

For a positive integer $n$ satisfying \eqref{eq:weight} for some $m\in
\Z$, let
\begin{multline}
S^W_i(n) =  \int_{\Gamma} S^W_i(z_1)S^W_i(z_2)\dots S^W_i(z_n)
 dz_1\dots dz_n: \pi^\ka_{\nu}\ra \pi^\ka_{\nu-n\alpha_i}
\label{eq:intertwiner3}
\end{multline}
where
\begin{equation}    \label{SiW}
S_{i}^W(z)= \; : \! e^{\int -\frac{1}{\kappa}b_i(z)dz}\!: \;
 : \; \pi^\ka_{\nu}\ra \pi^\ka_{\nu-\alpha_i}
\end{equation}
and $\Gamma\in H_{n}(Y_{n},\mc{L}_n(\nu,\kappa))$.

\medskip

Next, we find the action of the screening operators on the
cohomologies.

\begin{Lem}\label{Lem:screning-for-W}
Under the isomorphism \eqref{HW}, the map
  \eqref{eq:induced-screaning} is identified with the operator
$$
S^W_i(\lam_i+1): \pi^\ka_{\lam-\ka\check{\mu}} \to
\pi^\ka_{\lam-\ka\check{\mu}-(\lam_i+1)\alpha_i}.
$$
\end{Lem}
\begin{proof}
Let $\Phi'$ denote the isomorphism \eqref{HW}.
It follows from Proposition \ref{Pro:key-iso} that
$$\Phi'\circ  S_i(z)=(S_i(z) + :\!e^{\int
  -\frac{1}{\kappa}b_i(z)dz}\! :)\circ \Phi'.$$ This implies that the
the operator
$$
S_i(\lam_i+1): \Wak{\lam}{\ka} \to \Wak{\lam-(\lam_i+1)\alpha_i}{\ka}
$$
induces on the cohomologies a map
$$
\pi^\ka_{\lam-\ka\cmu} \to \pi^\ka_{\lam-\ka\cmu-(\lam_i+1)\alpha_i}
$$
equal to the operator
$S^W_i(\lam_i+1)$ plus the sum of operators with non-zero
weight with respect to the Cartan subalgebra. The latter sum gives
rise to the zero map on the cohomologies since both Fock
representations $\pi^\ka_{\lam-\ka\cmu}$ and
$\pi^\ka_{\lam-\ka\cmu-(\lam_i+1)\alpha_i}$ have zero weight.
\end{proof}

Now we are ready to prove Theorem \ref{Th:symmetry}.

\subsection{Completion of the proof of Theorem \ref{Th:symmetry}}

By Lemmas \ref{HW} and \ref{Lem:screning-for-W},
Wakimoto modules are acyclic with respect to the cohomology functor
$H^i_{DS,\cmu}(?)$ and $T^\kappa_{\lam,\cmu}$ is identified
with the 0th cohomology of a complex $\ol{C}^\bullet_\lam$ which
starts as follows:
\begin{equation}    \label{pi complex}
0 \ra \pi^\ka_{\lam-\ka\check{\mu}} \overset{\ol{d}^0_\lam}{\ra} \bigoplus_{i=1}^r
\pi^\ka_{\lam-\ka\check{\mu}-(\lam_i+1)\al_i} \ra \dots
\end{equation}
with
\begin{equation}    \label{bard0}
\ol{d}^0_\lam=\sum_{i=1}^{r} c_i S^W_i(\lam_i+1)
\end{equation}
obtained by applying the functor $H^0_{DS,\cmu}(?)$ to each term of
the resolution of Proposition
\ref{Pro:resolution-for-generic-level-more-general} and using Lemma
\ref{Lem:screning-for-W}. In Section \ref{higher} we will prove that
the higher cohomologies of the complex \eqref{pi complex} vanish for
irrational $\ka$. For now, we just focus on its 0th cohomology:
\begin{equation}    \label{Tlm}
T^\kappa_{\lam,\cmu}\cong \bigcap_{i=1}^r \on{Ker}_{\pi^\ka_{\lam-\ka\cmu}}
S^W_i(\lam_i+1).
\end{equation}

Let $\check{\pi}^{\check\ka}$ be the Heisenberg vertex algebra of
$\lh$ of level $\check{\kappa}$. It is generated by the fields $^L\neg
b_i(z)$, $i=1,\dots, \on{rank}\lg$, with the OPEs
\begin{align}
{}^L\neg b_i(z){}^L\neg b_j(w)\sim
\frac{\check{\kappa}({}^L\neg \alpha_i|{}^L\neg
  \alpha_j)}{(z-w)^2}
\label{eq:Lb_i},
\end{align}
where ${}^L\neg \alpha_i$ is the $i$th simple root of $\lg$ and
$\check\ka=1/m\ka$. Note that $(\cdot|\cdot)$ now stands for the inner
product on $(\lh)^*$ such that the square length of its maximal root
is equal to $2$.

According to \cite{FF} (see also \cite{Fre05,FreBen04}), the duality
\eqref{eq:FF-duality} is induced by the vertex algebra isomorphism
\begin{align} \notag
\pi^\ka &\isomap \check{\pi}^{\check\ka}, \\    \label{checkpi}
b_i(z)&\mapsto -m\frac{\kappa(\alpha_i|\alpha_i)}{2}{}^L\neg b_i(z), \\
\ka &\mapsto \check\ka = \frac{1}{m\ka}, \notag
\end{align}
where $m$ is the lacing number of $\g$,
that is,
the maximal number of the edges in the Dynkin diagram of $\g$.

In the same way as above, we obtain in the case of $\lg$ that
\begin{align*}
&\check{T}^\kappa_{\cmu,\lam} \cong \bigcap_{i=1}^r
\on{Ker}_{\pi^{\check\ka}_{\cmu-\check\ka\lam}}
\check{S}^W_i(\cmu_i+1).
\end{align*}
Therefore in order to prove Theorem \ref{Th:symmetry} it is sufficient
to establish the isomorphisms
\begin{equation}    \label{Kers}
\on{Ker}_{\pi^\ka_{\lam-\ka\cmu}}
S^W_i(\lam_i+1) \cong \on{Ker}_{\check\pi^{\check\ka}_{\cmu-\check\ka\lam}}
\check{S}^W_i(\cmu_i+1), \qquad i=1,\dots,r
\end{equation}
(for irrational $\ka$).

To prove the latter, observe that we have tensor product decompositions
$$
\pi^\ka = \pi^\ka_i \otimes \pi^{\ka\perp}_i, \qquad \pi^\ka_\nu =
\pi^\ka_{i,\nu_i} \otimes \pi^{\ka\perp}_{i,\nu^\perp},
$$
where $\pi^\ka_i$ is the Heisenberg vertex subalgebra generated by the
field $b_i(z)$ and $\pi^{\ka\perp}_i$ is its centralizer, which is a
Heisenberg vertex algebra generated by the fields orthogonal to
$b_i(z)$. We denote by $\pi^\ka_{i,\nu_i}$ and
$\pi^{\ka\perp}_{i,\nu^\perp}$ the corresponding modules. By
construction, the operator $S^W_i(\lam_i+1)$ commutes with
$\pi^{\ka\perp}_i \subset \pi^\ka$. Therefore
$$
\on{Ker}_{\pi^\ka_{\lam-\ka\cmu}} S^W_i(\lam_i+1) =
\pi^{\ka\perp}_{i,(\lam-\ka\cmu)^\perp} \otimes
\on{Ker}_{\pi^\ka_{i,\lam_i-\ka\cmu_i}} S^W_i(\lam_i+1).
$$

We have a similar decomposition in the case of $\lg$. Furthermore,
under the identification of the Heisenberg vertex algebras $\pi^\ka$
and $\check\pi^{\check\ka}$, the subalgebras $\pi^\ka_i$ and
$\pi^{\ka\perp}_i$ are identified with the corresponding subalgebras
$\check\pi^{\check\ka}_i$ and $\check\pi^{\check\ka \perp}_i$ of
$\check\pi^{\check\ka}$. We also have
$$
\on{Ker}_{\check\pi^{\check\ka}_{\cmu-\check\ka\lam}}
\check{S}^W_i(\cmu_i+1) =
\check\pi^{\check\ka\perp}_{i,(\cmu-\check\ka\lam)^\perp}
\otimes \on{Ker}_{\pi^{\check\ka}_{i,\cmu_i-\check\ka\lam_i}}
\check{S}^W_i(\cmu_i+1)
$$
Since $\pi^{\ka\perp}_{i,(\lam-\ka\cmu)^\perp} \cong
\check\pi^{\check\ka\perp}_{i,(\cmu-\check\ka\lam)^\perp}$,
the isomorphism \eqref{Kers}
is equivalent to the isomorphism
\begin{equation}    \label{Kersi}
\on{Ker}_{\pi^\ka_{i,\lam_i-\ka\cmu_i}}
S^W_i(\lam_i+1) \cong \on{Ker}_{\pi^{\check\ka}_{i,\cmu_i-\check\ka\lam_i}}
\check{S}^W_i(\cmu_i+1).
\end{equation}
The left hand side of \eqref{Kersi} is the kernel of the map
\begin{equation}    \label{KerVir}
S^W_i(\lam_i+1): \pi^\ka_{i,\lam_i-\ka\cmu_i} \to
\pi^\ka_{i,-2-\lam_i-\ka\cmu_i}.
\end{equation}
As shown in \cite{FF} (see the proof of Proposition 5, where the
notation $\nu$ corresponds to our $\ka^{1/2}$), it commutes with the
Virasoro algebra $\on{Vir}_i^\ka$ generated by the field
\begin{equation}    \label{Ti}
T_i(z) = \frac{1}{2\ka(\al_i|\al_i)} :\! b_i(z)^2 \! : +
\left( \frac{1}{(\al_i|\al_i)} - \frac{1}{2\ka} \right) \pa_z b_i(z).
\end{equation}
with central charge $c=13-6\gamma-6\gamma^{-1}$ where $\gamma =
\frac{2\ka}{(\al_i|\al_i)}$.

According to the results of \cite{Kac,FeFu,TsuKan86}, for irrational
$\ka$ (and hence $\gamma$), the kernel of the operator \eqref{KerVir}
is isomorphic to the irreducible module over the Virasoro algebra
\eqref{Ti} with lowest weight (lowest eigenvalue of $L_0$)
\begin{equation}    \label{Deltagamma}
\Delta^\gamma_{\lam_i,\cmu_i} = \gamma^{-1}\, \frac{\lam_i(\lam_i+2)}{4} +
\gamma \frac{\cmu_i(\cmu_i+2)}{4} - \frac{\lam_i\cmu_i+\lam_i+\cmu_i}{2},
\end{equation}
and the same is true for the kernel on the right hand side of
\eqref{Kersi}.

Thus, for irrational $\ka$ the isomorphisms \eqref{Kersi} hold for all
$i=1,\dots,r$, and hence so do the isomorphisms \eqref{Kers}. This
completes the proof.\qed

\section{Irreducibility and vanishing of higher
  cohomologies}\label{sec:proof}

The Miura map
$\Upsilon$ induces an injective homomorphism
\begin{align}
\Upsilon_{Zhu}:\on{Zhu}(\W^\kappa(\g))\hookrightarrow \on{Zhu}(\pi)=S(\mf{h}),
\end{align}
where $\on{Zhu}(V)$ is Zhu's algebra of $V$ (\cite{ACL17}).
For $\lam\in \h^*$
$$\chi(\lam):\text{(evaluation at $\lam$)} \circ \Upsilon_{Zhu}:\on{Zhu}(\W^\kappa(\g))\ra \C.$$
Then 
\begin{align}
\chi(\lam)=\chi(\cmu)\iff \lam+\rho-\kappa \check{\rho} \in W(\cmu+\rho-\kappa\check{\rho}).
\label{eq:hw}
\end{align}
Here, $\rho$ and $\check{\rho}$ are the half sum of the positive roots
and the positive coroots of $\g$, respectively.

\subsection{Irreducibility and vanishing for irrational
  $\ka$}    \label{irrvan}

Let $\mathbf{L}_{\chi(\lam)}^\kappa$ be the irreducible
representation of $\W^\kappa(\g)$ with highest weight
$\chi(\lam)$. Recall that $\mathbb{V}_{\nu}^\kappa$ denotes the
irreducible highest weight representation of $\affg_\kappa$ with
highest weight $\nu$.

The following assertion follows from  \cite{Ara04} and \cite[Theorem
9.14]{Ara07}.
\begin{Pro}    \label{irr}
Let $\kappa$ be irrational,
 $\lam\in P_+$, $\cmu\in P^{\vee}_+$.
Then
$H^i_{DS}(\mathbb{V}_{\lam-\kappa\cmu}^\kappa)=0$ for $i\ne 0$
and
$H^0_{DS}(\mathbb{V}_{\lam-\kappa\cmu}^\kappa) \cong
\mathbf{L}_{\chi(\lam-\kappa\cmu)}^\kappa$.
\end{Pro}

Recall the isomorphism \eqref{checkpi} between the Heisenberg algebras
$\pi$ and $\check\pi$ which induces the duality isomorphism
\eqref{eq:FF-duality}. It implies the following statement.

\begin{Lem}\label{Lem:symmetry}
Let $\lam\in P$,  $\cmu\in P^{\vee}$.
Under the duality isomorphism \eqref{eq:FF-duality}, we have
\begin{align*}
\mathbf{L}_{\chi(\lam-\kappa\cmu)}^\kappa\cong
\mathbf{L}_{\chi(\cmu-\check{\kappa}\lam)}^{\check{\kappa}}.
\end{align*}
\end{Lem}

The following assertion was conjectured by Creutzig and Gaiotto
\cite{CreGai}.

\begin{Th}\label{Th:CG}
Let $\kappa$ be irrational.
For any $\lam\in P_+$, $\cmu\in \check{P}_+$,
we have
$$T^{\kappa}_{\lam,\cmu}\cong \mathbf{L}_{\chi(\lam-\kappa\cmu)}^\kappa.$$
\end{Th}

\begin{Co}
The modules $T^{\kappa}_{\lam,\cmu}, \lam \in P_+, \cmu \in
\check{P}_+$ are irreducible for irrational $\ka$.
\end{Co}

This is the statement of Theorem \ref{Th:irr}.

In order to prove Theorem \ref{Th:CG}, we will need the following
generalization of Proposition \ref{Pro:resolution-weyl} which
has been proved in \cite{ACL17}.

\begin{Pro}\label{Pro:resolution-for-generic-level-more-general}
Let $\kappa$ be irrational,
$\lam\in P_+$, $\cmu \in \check{P}_+$.
There exists a resolution $C^\bullet_{\lam-\ka\cmu}$
of the $\affg$-module $\mathbb{V}_{\lam-\kappa\cmu}^\kappa$ of the form
\begin{align*}
&0\ra \mathbb{V}_{\lam-\kappa\cmu}^\kappa \ra
C^0_{\lam-\ka\cmu}\overset{d^0_{\lam-\ka\cmu}}{\ra}
C^1_{\lam-\ka\cmu}\ra\dots \ra C^n_{\lam-\ka\cmu}\ra 0,\\
& C^i_{\lam-\ka\cmu} =
\bigoplus_{w\in W\atop \ell(w)=i}\Wak{w\circ \lam-\kappa\cmu}{\ka}.
\end{align*}
The differential $d^0_{\lam-\ka\cmu}$ is given by
$$d^0_{\lam-\ka\cmu}=\sum_{i=1}^{r}
c_i S_i(\lam_i+1)$$
for some $c_i\in \C$, with $\lam_i=\bra \lam,\alpha_i^\vee\ket$.
\end{Pro}

\begin{proof}[Proof of Theorem  \ref{Th:CG}]
Let us apply the quantum Drinfeld--Sokolov
reduction functor (without twist by $\cmu$) to each term of the
resolution of Proposition
\ref{Pro:resolution-for-generic-level-more-general}. Then we find that
$H_{DS}^0(\V_{\lam-\kappa \cmu}^\kappa)$ is the 0th cohomology of the
complex obtained by applying the functor $H^i_{DS}(?)$ to the
resolution in Proposition
\ref{Pro:resolution-for-generic-level-more-general}. In the same way
as in the proof of Lemma \ref{Lem:screning-for-W} 
we then obtain that
\begin{equation}    \label{Lchi}
H_{DS}^0(\V_{\lam-\kappa \cmu}^\kappa)
\cong \bigcap_{i=1}^r \on{Ker}_{\pi_{\lam-\kappa\cmu}}
S_i^W(\lam_i+1).
\end{equation}
Combining the isomorphisms \eqref{Tlm} and \eqref{Lchi}, we obtain an
isomorphism
\begin{equation}    \label{twoDS}
T^\kappa_{\lam,\cmu} \cong H_{DS}^0(\V_{\lam-\kappa \cmu}^\kappa).
\end{equation}
According to Proposition \ref{irr},
$\mathbf{L}_{\chi(\lam-\kappa\cmu)}^\kappa\cong
H_{DS}^0(\V_{\lam-\kappa \cmu}^\kappa)$. Together with \eqref{twoDS},
this completes the proof of Theorem \ref{Th:CG}.
%
\end{proof}

Note that Theorem \ref{Th:symmetry} also follows from Lemma
\ref{Lem:symmetry} and Theorem \ref{Th:CG}. Thus, we obtain an
alternative proof of Theorem \ref{Th:symmetry}. Both proofs rely on
resolutions of irreducible $\ghat$-modules in terms of Wakimoto
modules. The proof given in the previous section uses in addition to
that an isomorphism of kernels of screening operators in the rank 1
case, which boils down to some properties of representations of the
Virasoro algebras. The proof presented in this section does not use
representations of the Virasoro algebra, but uses instead Proposition
\ref{irr} stating that $H^0_{DS}(\mathbb{V}_{\lam-\kappa\cmu}^\kappa)$
is irreducible.

\subsection{Cohomology vanishing for arbitrary
  $\ka$}    \label{vangen}

In this subsection we prove Theorem \ref{Th:van} by generalizing the
proof in the case $\lam=\cmu=0$ given in Sect.\ 15.2 of
\cite{FreBen04} (which followed \cite{dBT}).

We start by representing the complex $C({\mb V}_{\lam,\ka})$ as a tensor
product of two subcomplexes. Let $\{ J^a \}$ be a basis of $\g$ which
is the union of the basis
$\{ J^\al \}_{\al \in \Delta_+}$ of $\n_+$ (where $J^\al = e^\al$)
and a basis $\{ J^{\ol{a}} \}_{\ol{a} \in \Delta_- \cup I}$ of $\bb_-
= \n_- \oplus \h$ consisting of root vectors $f^\al, \al \in
\Delta_+$, in $\n_-$ and vectors $h^i, i \in I = \{
1,\ldots,\ell \}$, in $\h$. Thus, we use Latin upper indices to denote
arbitrary basis elements, Latin indices with a bar to denote elements
of $\bb_-$, and Greek indices to denote basis elements of $\n_+$.

Denote by $c^{ab}_d$ the structure constants of $\g$ with respect to
the basis $\{ J^a \}$.

Define the following currents:
\begin{equation}    \label{whJ}
\wh{J}^a(z) = \sum_{n\in\Z} \wh{J}^a_n z^{-n-1} = J^a(z) +
\sum_{\beta,\gamma \in \Delta_+} c^{a\beta}_\gamma : \! \psi_\gamma(z)
\psi^*_\beta(z)\! : \; .
\end{equation}

Now, the first complex, denoted by $C({\mb V}_{\lam,\ka})_0$, is spanned
by all monomials of the form
\begin{equation}    \label{basis of C0}
\wh{J}^{\ol{a}(1)}_{n_1} \ldots \wh{J}^{\ol{a}(r)}_{n_r}
\psi^*_{\al(1),m_1} \ldots \psi^*_{\al(s),m_s} v, \qquad v \in V_\lam
\end{equation}
(recall that $J^{\ol{a}} \in \bb_-$). The second complex, denoted
by $C({\mb V}_{\lam,\ka})'$, is spanned by all monomials of the form
$$
\wh{J}^{\al(1)}_{n_1} \ldots \wh{J}^{\al(r)}_{n_r}
\psi_{\al(1),m_1} \ldots \psi_{\al(s),m_s}
$$
(recall that $J^{\al} \in \n_+$). We have an analogue of formula
(15.2.3) of \cite{FreBen04}: the natural map
\begin{equation}    \label{dec as vs}
C({\mb V}_{\lam,\ka})' \otimes C({\mb
  V}_{\lam,\ka})_0 \overset{\sim}\longrightarrow C({\mb V}_{\lam,\ka})
\end{equation}
sending $A \otimes B$ to $A \cdot B$ is an isomorphism of graded
vector spaces.

We then have an analogue of Lemma 15.2.5 of \cite{FreBen04}: the
cohomology of $(C({\mb V}_{\lam,\ka}),d_\cmu)$ is isomorphic to the
tensor product of the cohomologies of the two complexes in \eqref{dec
  as vs}: $(C({\mb V}_{\lam,\ka})_0,d_\cmu)$ and $(C({\mb
  V}_{\lam,\ka})',d_\cmu)$. This is proved in the same way as in
\cite{FreBen04}, using the commutation relations established in
Sect. 15.2.4, in which we set $\chi=\Psi_\cmu$.

In the same way as in Sect. 15.2.6, we prove that the cohomology of
the complex $(C({\mb V}_{\lam,\ka})',d_\cmu)$ is one-dimensional, in
cohomological degree 0. Thus, we have an analogue of Lemma 15.2.7: the
cohomology of $(C({\mb V}_{\lam,\ka}),d_\cmu)$ is isomorphic to the
cohomology of its subcomplex $(C({\mb V}_{\lam,\ka})_0,d_\cmu)$.

To compute $H^\bullet(C({\mb V}_{\lam,\ka})_0,d_\cmu)$, we introduce a
double complex as in Sect. 15.2.8 of \cite{FreBen04}. The convergence
of the resulting spectral sequence is guaranteed by the fact that
$(C({\mb V}_{\lam,\ka})_0,d_\cmu)$ is a direct sum of
finite-dimensional subcomplexes obtained via the $\Z$-grading
introduced below in Section \ref{char form}. The $0$th differential is
$\hat{\Psi}_{\cmu}$. We have an analogue of formula (15.2.4) from
\cite{FreBen04}:
\begin{equation}    \label{p+}
[\hat{\Psi}_{\cmu},\wh{J}^{\ol{a}}_n] = \sum_{\beta\in \Delta_+\atop
  k\in \Z} ([\sigma_{\cmu}(p_-),J^{\bar a}_n]| J^\beta_{-k})\psi_{\beta,k+1}^*,
\end{equation}
where 
$$
p_-=\sum_{i=1}^{\ell}\frac{(\alpha_i,\alpha_i)}{2} f_i, 
$$
(here $f_i=f_i \cdot 1$), $\sigma_\cmu$ is the automorphism introduced
in Section \ref{twist}, and we use the notation
$$
(A t^n|B t^m) = \ka_0(A,B) \delta_{n,-m}.
$$
In \cite{FreBen04}, formula (15.2.4) (to which our formula \eqref{p+}
specializes when $\cmu=0$) was used to show that $\bb_- t^{-1}
\C[t^{-1}]$ has a basis consisting of the elements $P_i^{(n)}, n<0,
i=1,\ldots,r$, forming a basis of the Lie subalgebra
$$
\wh{\mathfrak a}_- = \on{Ker}
\on{ad}(p_-) \subset \bb_- t^{-1} \C[t^{-1}]
$$
and elements $I^\al_n, \al \in \Delta_+, n<0$, such that $\on{ad}(p_-)
\cdot I^\al_n = f_\al t^n$ (here $f_\al$ is a generator of the
one-dimensional subspace of $\n_-$ corresponding to the root $-\al$).

The existence of this basis is equivalent to the surjectivity of the map
\begin{equation}    \label{surj map0}
\on{ad} p_-: \bb_- t^{-1}\C[t^{-1}] \to \n_-[t^{-1}]
\end{equation}
which implies the following direct sum decomposition (as a
vector space)
\begin{equation} \label{decomp0} \bb_- t^{-1}\C[t^{-1}] =
  \wh{\mathfrak a}_- \oplus (\on{ad} p_-)^{-1}(\n_-[t^{-1}]),
\end{equation}
where the second vector space on the right hand side denotes a
particular choice of a subspace of $\bb_- t^{-1}\C[t^{-1}]$ that
isomorphically maps onto $\n_-[t^{-1}]$ under the map $\on{ad}
p_-$. This decomposition, in turn, implies that the complex $C({\mb
  V}^\ka_0)_0$ is isomorphic, as a vector space, to the tensor product
\begin{equation}    \label{tens pr0}
C({\mb V}^\ka_0)_0 =
  U(\wh{\mathfrak a}_-) \otimes \C[\wh{I}^\al_n]_{\al \in \Delta_+, n<0}
  \otimes \bigwedge(\psi^*_{\al,n})_{\al \in \Delta_+, n<0},
\end{equation}
where $\C[\wh{I}^\al_n]_{\al \in \Delta_+, n<0}$ stands for the linear span
of lexicographically ordered monomials in the $\wh{I}^\al_n$. The
differential $\chi = \wh\Psi_0$ acts as follows:
\begin{equation}    \label{psi0}
[\wh\Psi_0,\wh{P}^{(i)}_n] = 0, \qquad [\wh\Psi_0,\wh{I}^{\al}_n] =
  \psi^*_{\al,n+1}, \qquad [\wh\Psi_0,\psi^*_{\al,n}]_+=0.
\end{equation}
In Sect. 15.2.9 of \cite{FreBen04}, the decomposition \eqref{tens pr0}
and formulas \eqref{psi0} were used to show that the higher
cohomologies of the complex $C({\mb V}^\ka_0)_0$ vanish and the $0$th
cohomology is isomorphic to $U(\wh{\mathfrak a}_-)$. This proves the
vanishing of $H^j_{DS,0}(\V^\ka_0)$ for all $j \neq 0$.

We want to apply this argument for arbitrary $\lam \in P_+, \cmu \in
\check{P}_+$. In order to do that, we need to prove that the linear
map
\begin{equation}    \label{surj map}
\on{ad}(\sigma_\cmu(p_-)): \bb_- t^{-1}\C[t^{-1}] \to \n_-[t,t^{-1}]
\to \n_-[t,t^{-1}]/\n_-[t] \cong \n_- t^{-1}\C[t^{-1}],
\end{equation}
(which is the analogue of the map \eqref{surj map0} for general $\cmu$)
is surjective. To see that, let
\begin{equation}    \label{Ialncmu}
I^\al_{n,\cmu} = \sigma_\cmu(I^\al_{n-\bra \al,\cmu \ket}), \qquad n<0.
\end{equation}
Then the formula $\on{ad}(p_-) \cdot I^\al_{n,\cmu} = f_\al t^n$
implies that
$$
\on{ad}(\sigma_\cmu(p_-)) \cdot I^\al_{n,\cmu} = \sigma_{\cmu}(f_\al
t^{n-\bra \al,\cmu \ket}) = f_\al t^n.
$$
Moreover, $I^\al_m$ has the form
$$
I^\al_m = \sum_{i=1}^r b_i f_{\alpha-\alpha_i} t^m, \qquad b_i \in \C
$$
(in this formula, if $\alpha=\alpha_i$, then $f_{\alpha-\alpha_i}$
stands for the Cartan generator $h_i$). Therefore
$$
\sigma_\cmu(I^\al_m) = \sum_{i=1}^r b_i f_{\alpha-\alpha_i}
t^{m+\bra\alpha-\alpha_i,\cmu\ket}.
$$
Since $\cmu \in \check{P}_+$, it follows that the elements
$$
I^\al_{n,\cmu} = \sigma_\cmu(I^\al_{n-\bra \al,\cmu \ket}) =
\sum_{i=1}^r b_i f_{\alpha-\alpha_i} t^{n-\bra\alpha_i,\cmu\ket}
$$
with $n<0$ belong to $\bb_- t^{-1} \C[t^{-1}]$, and so the map
\eqref{surj map} is indeed surjective.

Therefore, we have the following analogue of the decomposition
\eqref{decomp0}
\begin{equation}    \label{decomp}
\bb_- t^{-1}\C[t^{-1}] = \wh{\mathfrak a}^\cmu_- \oplus (\on{ad}
\sigma_\cmu(p_-))^{-1}(\n_- t^{-1} \C[t^{-1}]),
\end{equation}
where $\wh{\mathfrak a}^\cmu_-$ is the kernel of the map \eqref{surj
  map}. This implies an analogue of the tensor product decomposition
\eqref{tens pr0}:
\begin{equation}    \label{tens pr}
C({\mb V}_{\lam,\ka})_0 =
  U(\wh{\mathfrak a}^\cmu_-) \otimes \C[\wh{I}^\al_{n,\cmu}]_{\al \in
    \Delta_+, n<0} \otimes \bigwedge(\psi^*_{\al,n})_{\al \in
    \Delta_+, n<0} \otimes V_\lam
\end{equation}
where $I^\al_{n,\cmu}, \al \in \Delta_+, n<0$, is defined by the
formula \eqref{Ialncmu} and $\C[\wh{I}^\al_{n,\cmu}]_{\al \in
  \Delta_+, n<0}$ stands for the linear span of lexicographically
ordered monomials in the $\wh{I}^\al_{n,\cmu}$. The differential
$\wh\Psi_\cmu$ acts as follows:
\begin{equation}    \label{psi1}
[\wh\Psi_\cmu,\wh{P}] = 0, \quad \forall P \in \wh{\mathfrak
  a}^\cmu_-,
\qquad [\wh\Psi_\cmu,\wh{I}^{\al}_{n,\cmu}] = \psi^*_{\al,n+1},
\end{equation}
\begin{equation}    \label{psi2}
[\wh\Psi_\cmu,\psi^*_{\al,n}]_+=0, \qquad \wh\Psi_\cmu \cdot v = 0,
\qquad \forall v \in V_\lam.
\end{equation}
In the same way as in Sect. 15.2.9 of \cite{FreBen04}, we then use the
decomposition \eqref{tens pr} and formulas \eqref{psi1}, \eqref{psi2}
to show that the higher cohomologies of the complex $(C({\mb
  V}_{\lam,\ka})_0,\wh\Psi_\cmu)$ vanish and the $0$th cohomology is
isomorphic to $U(\wh{\mathfrak a}^\cmu_-) \otimes V_\lam$. This
implies the statement of Theorem \ref{Th:van}.\qed

\subsection{Character formula}    \label{char form}

We define a $\Z_+$-grading on the complex $C({\mb V}_{\lam,\ka})$ as
follows: $\on{deg} v_{\lam,\ka} = 0$, where $v_\lam$ is the highest
weight vector of ${\mb V}_{\lam,\ka}$,
$$
\on{deg} e_\al t^n = \on{deg} \psi_{\al,n} = -n -
\bra\al,\cmu+\check\rho \ket, \qquad \on{deg} f_\al t^n = \on{deg}
\psi^*_{\al,n} = -n + \bra\al,\cmu+\check\rho \ket,
$$
$$
\on{deg} h_{i,n} = -n.
$$
We find that $\on{deg} d_{\on{st}} = \on{deg} \wh\Psi_\cmu = 0$, so
the differential $d_\cmu$ preserves the grading and the complex
$(C({\mb V}_{\lam,\ka}),d_\cmu)$ decomposes into a direct sum of
homogeneous subcomplexes corresponding to all non-negative
degrees. The same is true for the subcomplex $(C({\mb
  V}_{\lam,\ka})_0,d_\cmu)$.

It is easy to see that the homogeneous subcomplexes of $(C({\mb
  V}_{\lam,\ka})_0,d_\cmu)$ are finite-dimensional. Hence we can use
this $\Z_+$-grading and the vanishing Theorem \ref{Th:van} to find the
character of $T^\ka_{\lam,\cmu}$, which appears as the $0$th
cohomology of $(C({\mb V}_{\lam,\ka})_0,d_\cmu)$, by the taking the
alternating sum of characters of the $j$th terms of $C({\mb
  V}_{\lam,\ka})_0$:
\begin{align*}
\on{char} T^\ka_{\lam,\cmu} &= \sum_{j \geq 0} (-1)^j \on{char}
C^j({\mb V}_{\lam,\ka})_0 \\
&= \on{char}_\cmu V_\lam \cdot \prod_{\al \in \Delta_+\atop n\geq
  \bra\al,\cmu+\check\rho \ket} (1-q^n) \prod_{\al \in \Delta_+\atop n>
  \bra\al,\cmu+\check\rho \ket} (1-q^n)^{-1} \prod_{n>0} (1-q^n)^{-r}
\\
&= \on{char}_\cmu V_\lam \cdot \prod_{\al \in \Delta_+}
(1-q^{\bra\al,\cmu+\check\rho \ket}) \prod_{n>0} (1-q^n)^{-r}.
\end{align*}
Here $\on{char}_\cmu V_\lam$ is the character of the
finite-dimensional representation $V_\lam$ with respect to the
$\Z_+$-grading defined by the formulas $\on{deg} v_\lam = 0$, where
$v_\lam$ is the highest weight vector of $V_\lam$, and $\deg f_\al =
\bra\al,\cmu+\check\rho \ket$.

By the Weyl character formula,
$$
\on{char}_\cmu V_\lam = q^{\bra \lam+\rho,\cmu+\check\rho \ket}
\sum_{w \in W} (-1)^{\ell(w)} q^{-\bra w(\lam+\rho),\cmu+\check\rho \ket} 
\prod_{\al \in \Delta_+}
(1-q^{\bra\al,\cmu+\check\rho \ket})^{-1}.
$$
Therefore we obtain the following character formula for
$T^\ka_{\lam,\cmu}$ (for any $\ka \in \C$):
\begin{equation}    \label{char1}
\on{char} T^\ka_{\lam,\cmu} = q^{\bra \lam+\rho,\cmu+\check\rho \ket}
\sum_{w \in W} (-1)^{\ell(w)} q^{-\bra w(\lam+\rho),\cmu+\check\rho \ket}
\prod_{n>0} (1-q^n)^{-r}.
\end{equation}
It is independent of $\ka$ and clearly symmetrical under the exchange
of $\lam$ and $\cmu$ (as well as $\rho$ and $\check\rho$).

\subsection{Failure of Theorem \ref{Th:symmetry} for rational
  $\kappa$}    \label{fail}

In this subsection we show that the statement of Theorem
\ref{Th:symmetry} with rational $\kappa$ is false already for
$\g=\mf{sl}_2$. In this case, we will use the parameter
$\gamma=\ka/\ka_0 \in \C$ (then $\check\ka$ corresponds to
$\gamma^{-1}$, $\ka_c$ to $\gamma=-2$ and $\ka_{\mf{sl}_2}$ to
$\gamma=4$), and will identity weights $\lam \in P$ with the integers
$\bra \check\alpha,\lam \ket \in \Z$, coweights $\cmu \in \check{P}$
with the integers $\bra \cmu,\alpha \ket \in \Z$. It is proved in
\cite{{Ara05}} that for any complex $\gamma \neq -2$, the cohomology
$H_{DS}^0(?)$ defines an exact functor from the category
$\mc{O}_\kappa$ of $\affg_{\kappa}$-modules to the category $\mc{O}$
of modules over the Virasoro algebra with the central charge
$13-6\gamma-6\gamma^{-1}$. It sends the Verma module
$\mathbb{M}^\ka_{\lam}$ over $\affg_{\kappa}$ with highest weight
$\lam$ (resp.\ the contragradient dual $D(\mathbb{M}^\ka_{\lam})$ of
$\mathbb{M}^\ka_{\lam}$; resp.\ the unique simple quotient
$\mathbb{L}^\ka_\lam$ of $\mathbb{M}^\ka_{\lam}$) to the Verma module
(resp.\ the contragradient dual of the Verma module; resp.\ a simple
module or zero module) over the Virasoro algebra with lowest weight
(i.e. the lowest eigenvalue of the element $L_0$)
$\Delta^\gamma_{\lam,0} = \lam(\lam+2)/4\kappa-\lam/2$ (compare with
formula \eqref{Deltagamma}).

In particular, $T^{\kappa}_{\lam,0}=H_{DS}(\mathbb{V}^\ka_{\lam})$ is
a quotient of the Verma module $H_{DS}^0( \mathbb{M}^\ka_{\lam})$ and
hence is a cyclic module over the Virasoro algebra, generated by its
lowest weight vector.

Now suppose that $\gamma<0$. It is proved in \cite{Fre92} that in this
case the Wakimoto module $\Wak{\lam}{\ka}$ with $\lam\in \Z_+$ is
isomorphic to the contradradient dual $D(\mathbb{M}^\ka_{\lam})$ of
the Verma module $\mathbb{M}^\ka_{\lam}$ over $\affg_{\kappa}$ with
highest weight $\lam$, and that $H_{DS,0}^0(\Wak{\lam}{\ka})\cong
\pi_{\lam}^\kappa$ is isomorphic to the contradradient dual of the
corresponding Verma module over the Virasoro algebra. Thus,
$\pi_{\lam}^\kappa$ is a cocyclic module over the Virasoro algebra for
any $\lam\in \Z_+$.

In our counterexample, we will set $\gamma=-2$, $\lam=2$,
$\check{\mu}=0$. (Similar counterexamples can also be obtained for any
negative integer $\gamma \leq -2$ and $\lam$ from an infinite
subset of $\Z_+$ depending on $\gamma$.) Then we have
$\mathbb{V}^\ka_0\cong \mathbb{L}^\ka_0$ and there is an exact
sequence
$$
0\ra \mathbb{L}^\ka_0\ra \mathbb{V}^\kappa_{2}\ra
\mathbb{L}^\kappa_{2}\ra 0
$$
(see e.g. \cite{Mal90,KasTan95}). Applying the functor $H_{DS}^0(?)$,
we get an exact sequence
\begin{equation}    \label{exact1}
0\ra \mathbf{L}_{\chi(0)}^\kappa \ra T^\kappa_{2,0}\ra 
\mathbf{L}_{\chi(2)}^\kappa\ra 0.
\end{equation}
The $L_0$-lowest weights of $\mathbf{L}_{\chi(0)}^\kappa$ and
$\mathbf{L}_{\chi(2)}^\kappa$ are $0$ and $-2$, respectively.
Therefore the image of $\mathbf{L}_{\chi(0)}^\kappa$ in
$T^\kappa_{2,0}$ is generated by a singular vector of weight
$2$. Thus, the module $T^\kappa_{2,0}$ is a cyclic module over the
Virasoro algebra, generated by its lowest weight vector, which is an
extension of the irreducible module $\mathbf{L}_{\chi(2)}^\kappa$ by
the irreducible module $\mathbf{L}_{\chi(0)}^\kappa$.

Next, consider
$T^{\check{\kappa}}_{0,2}=H_{DS,2}^0(\mathbb{V}^{\check{\kappa}}_0)$.
Our character formula \eqref{char1} shows that
$T^{\check{\kappa}}_{0,2}$ and $T^\kappa_{2,0}$ have the same
characters. Therefore, their irreducible subquotients are also the
same. However, we will now show that these two modules are not
isomorphic to each other.

The embedding
$\mathbb{V}^{\check{\kappa}}_0 \hookrightarrow
\mathbb{W}_{\lam}^{\check{\kappa}}$ induces a map
\begin{equation}    \label{Ttopi}
T^{\check{\kappa}}_{0,2}=H_{DS,2}^0(\mathbb{V}^{\check{\kappa}}_0) \to
H_{DS,2}^0(\mathbb{W}_{\lam}^{\check{\kappa}})
=\pi^{\check{\kappa}}_{-2\check{\kappa}}.
\end{equation}
With our choice of $\kappa$, it follows from \cite{Fre92} that
$\pi^{\check{\kappa}}_{-2\check{\kappa}}$ is a cocyclic module over
the Virasoro algebra, generated by its lowest weight vector. Its
character coincides with the character of $\pi^{\kappa}_{2}$, and
hence it is isomorphic to $\pi^{\kappa}_{2}$.

We claim that the map \eqref{Ttopi} is injective. This does not follow
immediately since we don't know whether $H_{DS,2}^0(?)$ is an exact
functor. However, we know from the character formula \eqref{char1}
that the weight $2$ subspace of $T^{\check{\kappa}}_{0,2}$ is
$2$-dimensional. Furthermore, it is clear that the images of
$\hat{h}_{-2}v$ and $\hat{h}_{-1}^2v$ (where $\{e,h,f\}$ is the
standard basis of $\mf{sl}_2$ and $v$ is the highest weight vector of
$C(\mathbb{V}^{\check{\kappa}}_0)$) in
$T^{\check{\kappa}}_{0,2}=H_{DS,2}^0(\mathbb{V}^{\check{\kappa}}_0)$
are linearly independent. Hence they form a basis of this weight $2$
subspace. But the map \eqref{Ttopi} sends these vectors to non-zero
scalar multiples of the vectors $b_{-2}v_{-2\check{\kappa}}$ and
$b_{-1}^2v_{-2\check{\kappa}}$, which form a basis in the weight 2
subspace of $\pi^{\check{\kappa}}_{-2\check{\kappa}}$. Therefore, the
map \eqref{Ttopi} is injective on the weight $2$ subspaces. But
$T^{\check{\kappa}}_{0,2}$ has the same irreducible subquotients as
$T^\ka_{2,0}$, i.e. the ones with lowest weights $0$ and $2$ (see the
exact sequence \eqref{exact1}). From the injectivity on the weight $2$
subspaces, it then follows that the map \eqref{Ttopi} itself is
injective.

Recalling that $\pi^{\check{\kappa}}_{-2\check{\kappa}}$ is a cocyclic
module over the Virasoro algebra, we then find that
$T^{\check{\kappa}}_{0,2}$ is cocyclic as well. Therefore, we have a
non-trivial extension
\begin{equation}    \label{exact2}
0\ra \mathbf{L}_{\chi(2)}^\kappa \ra T^{\check{\kappa}}_{0,2} \ra 
\mathbf{L}_{\chi(0)}^\kappa\ra 0.
\end{equation}
Comparing the extensions \eqref{exact1} and \eqref{exact2}, we
conclude that $T^{\check{\kappa}}_{0,2}$ is {\em not} isomorphic to
$T^\kappa_{2,0}$. Rather, $T^{\check{\kappa}}_{0,2}$ is isomorphic to
a different module: the contragradient dual of $T^\kappa_{2,0}$. Thus,
we obtain a counterexample to the statement of Theorem
\ref{Th:symmetry} with rational $\kappa$.

\section{Resolutions and vanishing}    \label{higher}

In this section, we will give a more detailed description of the
complexes obtained by applying the $\cmu$-twisted Drinfeld--Sokolov
reduction functor $H^\bullet_{DS,\cmu}(?)$ to the resolution of the
Weyl module $\V^\ka_\lam$ described in Proposition
\ref{Pro:resolution-for-generic-level-more-general}. In our proof of
Theorem \ref{Th:symmetry}, we focused on the $0$th differential and
the $0$th cohomology of this complex, which is the module
$T^\kappa_{\lam,\cmu}$. Here, we will give formulas for the higher
differentials and will explain the connection to the BGG resolutions
of irreducible finite-dimensional representations of the corresponding
quantum groups, following \cite{FF93,FF:wak}. This works for all
irrational values of $\ka$.

Theorem \ref{Th:van} then implies that for irrational $\ka$ this
complex is a resolution of the ${\mc W}^\ka(\g)$-module
$T^\kappa_{\lam,\cmu}$ by Fock representations.  As an application, we
will write in Section \ref{char} the character of the module
$T^\kappa_{\lam,\cmu}$ as an alternating sum of characters of the
Fock representations appearing in the resolution. This reproduces the
character formula from Section \ref{char form}.

Finally, in Section \ref{limit} we will give an alternative proof, for
generic $\ka$, that the higher cohomologies of this complex (and hence
$H^j_{DS,\cmu}({\mb V}^\ka_\lam)$ with $j \neq 0$) vanish. It relies
on the vanishing of the higher cohomologies in the classical limit
$\ka \to \infty$. In this limit, the screening operators satisfy the
Serre relations of the Lie algebra $\g$, i.e. they generate an action
of the Lie subalgebra $\n_- \subset \g$. The cohomologies of our
complex in the limit $\ka \to \infty$ are therefore the cohomologies
of $\n_-$ acting on the $\ka \to \infty$ limit of the Fock
representation $\pi^\ka_\lam$. It is easy to show that this action is
co-free, so that higher cohomologies vanish. The vanishing of higher
cohomologies in the limit $\ka \to \infty$ implies the vanishing for
generic $\ka$ as well. This is a generalization of the argument that
was used in \cite{FF93}, which corresponds to the case $\lam=0,
\cmu=0$.

\subsection{Recollections from \cite{FF93}}    \label{recoll}

Using the results of the earlier works \cite{BMP,SV,Varchenko}, Feigin
and one of the authors showed in \cite{FF93} how to associate linear
operators between Fock representations to singular vectors in Verma
modules over the quantum group. Let us briefly recall this
construction.

Let $q=e^{\pi i/\ka}$ and $U_q(\g)$ the Drinfeld--Jimbo quantum group
with generators $e_i, K_i, f_i, i=1,\dots,r$ and standard relations
(see, e.g., \cite{FF93}, Sect. 4.5.1). Let $U_q(\n_-)$ (resp.,
$U_q(\b_+)$) be the lower nilpotent (resp., upper Borel) subalgebra of
$U_q(\g)$, generated by $f_i$ (resp, $K_i, e_i$) where
$i=1,\dots,r$. The generators $f_i$ satisfy the $q$-Serre relations
\begin{equation}    \label{quantumserre}
(\on{ad}_q f_i)^{-a_{ij}+1} \cdot f_j = 0,
\end{equation}
where $(a_{ij})$ is the Cartan matrix of $\g$. The notation
$\on{ad}_q f_i$ means the following: introduce a grading on the free
algebra with generators $e_i, i=1,\dots,l,$ with respect to the root
lattice $Q$ of $\g$, by putting $\deg f_i = -\al_i$. If $x$ is a
homogeneous element of this algebra of weight $\gamma \in Q$,
put $$\on{ad}_q f_i \cdot x = f_i x - q^{(\al_i|\gamma)} x f_i.$$

Next, we define Verma modules over $U_q(\g)$ as follows. Let $\C_\la$
be the one-dimensional representation of $U_q(\b_+)$, which is spanned
by a vector ${\mathbf 1}_\la$, such that $$e_i \cdot {\mathbf 1}_\la =
0, \quad \quad K_i \cdot {\mathbf 1}_\la = q^{(\la|\al_i)} {\mathbf
  1}_\la, \quad \quad i=1,\dots,r.$$

The Verma module $M_\la^q$ over $U_q(\g)$ of highest weight $\la$ is
the module induced from the $U_q(\b_+)$-module
$\C_\la$: $$M_\la^q = U_q(\g) \otimes_{U_q(\b_+)} \C_\la.$$
It is canonically isomorphic to $U_q(\n_-) {\mathbf 1}_\la$, and hence to
$U_q(\n_-)$.

Roughly speaking, the screening operators $Q_i = \int S^W_i(z) dz$,
where $S^W_i(z)$ is given by formula \eqref{SiW}, satisfy the
$q$-Serre relations \eqref{quantumserre} and hence generate
$U_q(\n_-)$. However, because of the multivalued nature of the OPEs
between the fields $S^W_i(z)$:
$$
S^W_i(z) S^W_j(w) = (z-w)^{(\al_i|\al_j)/\ka} : \! S^W_i(z) S^W_j(w)
\! :
$$
and the factor $z^{(\lam|\al_i)/\ka}$ appearing in the expansion of
$S^W_i(z)$ acting from $\pi^\ka_\lam$ to $\pi^\ka_{\lam-\al_i}$, a general
element of $U_q(\n_-)$, when expressed in terms of the screening
operators $Q_i$, is not well-defined as a linear operator between Fock
representations. Only those elements are well-defined for which there
is a non-trivial integration cycle on the corresponding configuration
space (of the variables of the currents $S^W_i(z)$ that have to be
integrated) with values in a one-dimensional local system. Such an
integration cycle, in turn, exists if and only if the element of
$U_q(\n_-)$, when viewed as a vector in $M^q_\lam$ (where $\lam$ is
the highest weight of the Fock representation from which we want our
operator to act), is a singular vector, i.e. is annihilated by the
generators $e_i, i=1,\dots,r$.

\medskip

{\em Remark on notation:} Our Heisenberg algebra generators $b_{i,n}$
correspond to $\beta^{-2} b_{i,n}$ of \cite{FF93}, and our $\ka$
corresponds to $\beta^{-2}$. However, we have a different sign in the
definition of the screening currents $S^W_i(z)$ (see formula
\eqref{SiW}) compared to \cite{FF93}, and for this reason our
$\pi^\ka_\nu$ corresponds to $\pi_{-\nu \beta}$ of \cite{FF93}. In
addition, our $U_q(\n_-)$ corresponds to $U_q(\n_+)$ of \cite{FF93},
for the same reason. Apart from this sign change, our notation is
compatible with that of \cite{FF93}.

\medskip

According to Lemma 4.6.6 of \cite{FF93}, we have the following result.

\begin{Lem}    \label{linop}
  Let $P \in U_q(\n_-)$ be such that $P \cdot {\mathbf 1}_\nu$ is a
  singular vector of $M_\nu^q$ of weight $\nu-\gamma$. Then for
  irrational $\ka$ the operator $V_P^\ka$ defined by formula (4.6.1) of
  \cite{FF93} (with $\beta=\ka^{-1/2}$) is a well-defined homogeneous
  linear operator $\pi^\ka_\nu \to \pi^\ka_{\nu-\gamma}$.
\end{Lem}

For example, let $P=f_i^n$, where $n \in \Z_+$. Then $P {\mathbf 1}_\nu$
is a singular vector in $M^q_\nu$ if $\nu$ satisfies equation
\eqref{eq:weight} for some $m \in \Z$. The corresponding operator
$V_P^\ka: \pi^\ka_\nu \to \pi^\ka_{\nu-n\al_i}$ is the operator $S^W_i(n)$
given by formula \eqref{eq:intertwiner3}.

Denote by $F^\bullet_\ka(\g)$ the complex $F^*_\beta(\g)$ constructed
in Sect. 4.6 of \cite{FF93}, where $\beta=\ka^{-1/2}$. It consists of
Fock representations and its differentials are constructed using the
BGG resolution of the trivial representation of $U_q(\g)$ and Lemma
\ref{linop}. It was proved in Theorem 4.6.9 of \cite{FF93} that the
$0$th cohomology of $F^\bullet_\ka(\g)$ is the ${\mc W}$-algebra ${\mc
  W}^\ka(\g)$ and all other cohomologies vanish for generic $\ka$.

This will be our complex corresponding to $\lam=0, \cmu=0$. And now we
construct a similar complex $F^\bullet_{\lam,\cmu,\ka}(\g)$ for all
$\lam \in P_+, \cmu \in \check{P}_+$. We will show that the $0$th
cohomology of $F^\bullet_{\lam,\cmu,\ka}(\g)$ is the ${\mc
  W}^\ka(\g)$-module $T_{\lam,\cmu}^\ka$ and all other cohomologies
vanish for irrational $\ka$.

\subsection{Generalization to non-zero $\lam$ and $\cmu$}

First, we generalize the complex to an arbitrary $\lam \in P_+$ and
$\cmu=0$. Consider the BGG resolution $B^{q,\lam}_\bullet(\g)$ of the
irreducible finite-dimensional representation $L^q_\lam$ of $U_q(\g)$
with highest weight $\lam \in P_+$ (see Remark 4.5.7 of
\cite{FF93}). Its degree $j$ part is
$$
B^{q,\lam}_j(\g) = \bigoplus_{\ell(w)=j} M_{w \circ \lam}^q, \qquad w
\circ \lam = w(\lam+\rho)-\rho.
$$
The differential is constructed in the same way as in Sect. 4.5.6 of
\cite{FF93} in the case $\lam=0$: For any pair $w, w'$ of elements of
the Weyl group of $\g$, such that $w\prec w''$, we have the embeddings
$i_{w',w}^q: M_{w' \circ \lam}^q \to M_{w \circ \lam}^q$
satisfying $i_{w'_1,w}^q \circ i_{w'',w'_1}^q = i_{w'_2,s}^q \circ
i_{w'',w'_2}^q$. The differential $d_j^{q,\lam}: B_j^{q,\lam}(\g) \to
B_{j-1}^{q,\lam}(\g)$ is given by the formula
\begin{equation}    \label{qdiff}
d_j^{q,\lam} = \sum_{\ell(w)=j-1,\ell(w')=j,w\prec w'} \epsilon_{w',w} \cdot
i_{w',w}^q.
\end{equation}
The embedding $i_{w',w}^q$ is given by the formula $u {\mathbf
  1}_{w'\circ \lam} \to u P^q_{w',w} {\mathbf 1}_{w \circ \lam}^q,
\forall u \in U_q(\n_-)$, where $P^q_{w',w} {\mathbf 1}_{w \circ
  \lam}^q$ is a singular vector in $M^q_{w \circ \lam}$ of weight
$w'\circ \lam$.

Now we use this BGG resolution to construct a complex
$F^\bullet_{\lam,0,\ka}(\g)$ as in Sect. 4.6.8 of \cite{FF93}. Namely,
we set
$$
F^j_{\lam,0,\ka}(\g) = \bigoplus_{\ell(w)=j} \pi^\ka_{w \circ \lam}
$$
and define the differential of this complex using the differential of
$B_\bullet^{q,\lam}(\g)$ by formulas analogous to formula (4.6.5) of
\cite{FF93}:
\begin{equation}   \label{qdiff1}
d^j_\lam = \sum_{\ell(w)=j-1,\ell(w')=j,w\prec w'} \epsilon_{w',w} \cdot
V_{P_{w',w}^q}^\ka.
\end{equation}
The nilpotency of this differential follows in the same
way as in the case $\lam=0$ \cite{FF93}. Furthermore, it follows from
the construction that the $0$th differential of the complex
$F^\bullet_{\lam,\ka}$ (recall that $s_i\circ \lam =
\lam-(\lam_i+1)\al_i$)
\begin{equation}    \label{zeroth diff}
d^0_\lam: \pi^\ka_\lam \to \bigoplus_{i=1}^r
\pi^\ka_{\lam-(\lam_i+1)\al_i}
\end{equation}
is equal to
\begin{equation}    \label{sum scr}
d^0_\lam = \sum_{i=1}^{r} v_i S^W_i(\lam_i+1),
\end{equation}
where $v_i \in \C^\times$ (compare with formulas \eqref{pi complex}
and \eqref{bard0}). The factors $v_i$ occur because our choice of
integration cycle $\Gamma$ in Theorem \ref{Th:Tsuhiya-Kanie} is {\em a
  priori} different from that of \cite{FF93}. Since the corresponding
cohomology group is one-dimensional, the resulting integrals are
proportional to each other, and $v_i$ is the corresponding
proportionality factor.

\medskip

Finally, we consider arbitrary $\cmu \in \check{P}_+$. We define the
complex $F^\bullet_{\lam,\cmu,\ka}(\g)$ as follows:
$$
F^j_{\lam,\cmu,\ka}(\g) = \bigoplus_{\ell(w)=j} \pi^\ka_{w \circ \lam - \ka\cmu}
$$
and define the differentials by the same formula as for the complex
$F^\bullet_{\lam,\ka}(\g)$.

In particular, the $0$th differential $d^0_{\lam,\cmu}$ equals the
differential \eqref{bard0} (up to the inessential scalar multiples in
front of the summands), and therefore we find that the $0$th
cohomology of our complex $F^\bullet_{\lam,\cmu,\ka}(\g)$ is
the ${\mc W}^\ka(\g)$-module $T^\ka_{\lam,\cmu}$.

\begin{Th}    \label{Th:res}
Let $\ka$ be irrational. Then we have

(1) The $j$th cohomology of the complex
$F^\bullet_{\lam,\cmu,\ka}(\g)$ is isomorphic to
$H^j_{DS,\cmu}({\mathbb V}^\ka_\lam)$.

(2) The $j$th cohomology of $F^\bullet_{\lam,\cmu,\ka}(\g)$ is
$T^\ka_{\lam,\cmu}$ if $j=0$ and $0$ if $j>0$.
\end{Th}

\subsection{Proof of Theorem \ref{Th:res}}

We will construct explicitly the higher differentials of the complex
\eqref{resWak}, which is a resolution of the Weyl module ${\mb
  V}^\ka_\lam$ in terms of the Wakimoto modules. This has already been
done in \cite{F:cargese,FF:wak} in the case $\lam=0$ and the
construction generalizes in a straightforward fashion to arbitrary
$\lam \in P_+$.

Recall that
$$
C^j_\lam=\bigoplus_{w\in W\atop \ell(w)=j}\Wak{w\circ \lam}{\ka}.
$$
Thus, the weights of the Wakimoto modules appearing in $C^j_\lam$ are
the same as those of the Verma modules appearing in the $j$th term
$B^{q,\lam}_j(\g)$ of the BGG resolution of $L^q_\lam$. We define the
differentials of the complex $C^\bullet_\lam$ by the above formula
\eqref{qdiff1}, in which we however use a different definition of
$V_P^\ka$. While in the definition of \cite{FF93}, which is used above
in formula \eqref{qdiff1}, $V_P^\ka$ is constructed using the ${\mc
  W}$-algebra screening currents $S^W_i(z)$, now we use in their place
the affine Kac--Moody screening currents $S_i(z)$ given by formula
\eqref{Si}. Let us denote the corresponding operator by
$\wt{V}_P^\ka$.

The fact that an analogue of Lemma \ref{linop} holds for these
screening currents was established in Sect. 3 of \cite{FF:wak}. This
implies that with this definition, we indeed obtain a
complex. Furthermore, for irrational $\ka$ we have $\Wak{w\circ
  \lam}{\ka} \cong M^{*\ka}_{w\circ \lam}$, as shown in the proof of
Proposition \ref{Pro:resolution-weyl}. Therefore we find that the
space of intertwining operators between $\Wak{w\circ \lam}{\ka} \to
\Wak{w'\circ \lam}{\ka}$ with $\ell(w)=j-1,\ell(w')=j,w\prec w'$ is
one-dimensional. We also know that each operator $\wt{V}_P^\ka$ is
non-zero because this is so in the limit $\ka \to \infty$, as
explained in Sect. 4 of \cite{FF:wak}. Therefore the complex
constructed this way is indeed isomorphic to the complex from
Proposition \ref{Pro:resolution-weyl}.

Now we apply to this complex the functor
$H^\bullet_{DS,\cmu}(?)$. According to Lemma \ref{HW}, we have
$H^\bullet_{DS,\cmu}(\Wak{w\circ \lam}{\ka}) \cong \pi^\ka_{w\circ
  \lam - \ka\cmu}$, so as a graded vector space, the complex we
obtain is precisely the complex
$F^\bullet_{\lam,\cmu,\ka}(\g)$. Furthermore, in the same way as in
the proof of Lemma \ref{Lem:screning-for-W} we obtain that the
corresponding differentials are given by the same formulas as the
differentials of the complex $C^\bullet_\lam$ but we have to replace
the Kac--Moody screening currents $S_i(z)$ by the ${\mc W}$-algebra
screening currents $S^W_i(z)$. Thus, we obtain precisely the
differentials \eqref{qdiff1} of the complex
$F^\bullet_{\lam,\cmu,\ka}(\g)$.

This proves part (1) of Theorem \ref{Th:res}. Part (2) now follows
from Theorem \ref{Th:van} and the definition of $T^\ka_{\lam,\cmu}$.

It is worth noting that the complex $F^\bullet_{\lam,\cmu,\ka}(\g)$
can be obtained in two ways: by applying the functor
$H^\bullet_{DS,\cmu}(?)$ to the resolution $C^\bullet_\lam$ of
$V^\ka_\lam$ (as above), and by applying the functor
$H^\bullet_{DS}(?)$ to the resolution $C^\bullet_{\lam-\ka\cmu}$ of
$V^\ka_{\lam-\ka\cmu}$ from Proposition
\ref{Pro:resolution-for-generic-level-more-general}. The second way
implies that its higher cohomologies vanish because of Proposition
\ref{irr}. Hence we obtain another proof of part (2) of Theorem
\ref{Th:res}.

\subsection{Character formula}    \label{char}

By definition, the character of a ${\mc W}^\ka(\g)$-module $M$ is
$\on{ch}(M) = \on{Tr}_M q^{L_0}$, where $L_0$ is the grading operator
obtained from the Virasoro generator $T(z)$ of ${\mc
  W}^\ka(\g)$. Theorem \ref{Th:res} implies that
$$
\on{ch}(T^\ka_{\lam,\cmu}) = \sum_{w \in W} (-1)^{\ell(w)}
\on{ch}(\pi^\ka_{w \circ \lam - \ka\cmu}).
$$
Next, according to the formula for $T(z)$ given in Sect. 4 of \cite{FF},
$$
\on{ch}(\pi^\ka_{\nu - \ka\cmu}) = q^{\Delta^\ka_{\nu,\cmu}} \prod_{n>0}
(1-q^n)^{-r},
$$
where
\begin{equation}    \label{Deltaka}
\Delta^\ka_{\nu,\cmu} = \frac{1}{2\ka}(\nu|\nu+2\rho) +
\frac{\ka}{2}(\cmu|\cmu+2\check\rho) - \langle
\nu+\rho,\cmu+\check\rho \rangle + \langle \rho,\check\rho \rangle.
\end{equation}
We also find that for every $w \in W$,
$$
\Delta^\ka_{w \circ \lam,\cmu} = \wt\Delta^\ka_{\lam,\cmu} - \bra
w(\lam+\rho),\cmu+\check\rho \ket,
$$
where
$$
\wt\Delta^\ka_{\lam,\cmu} = \frac{1}{2\ka}(\lam|\lam+2\rho) +
\frac{\ka}{2}(\cmu|\cmu+2\check\rho) + \bra \rho,\check\rho \ket.
$$
Therefore
\begin{equation}    \label{char2}
\on{ch}(T^\ka_{\lam,\cmu}) = q^{\wt\Delta^\ka_{\lam,\cmu}} \prod_{n>0}
(1-q^n)^{-r} \sum_{w \in W} (-1)^{\ell(w)} q^{-\bra
w(\lam+\rho),\cmu+\check\rho \ket}.
\end{equation}
Note that the eigenvalues of $L_0$ coincide with the
$\Z_+$-grading introduced in Section \ref{char form} up to a shift by
$\Delta^\ka_{\lam,\cmu}$ given by formula \eqref{Deltaka}. Hence
formula \eqref{char2} is equivalent to formula \eqref{char1}.

\subsection{The limit $\ka \to \infty$}    \label{limit}

In order to pass to the limit $\ka \to \infty$, we redefine the
complex $F^\bullet_{\lam,\cmu,\ka}(\g)$ slightly. Define the complex
$\wt{F}^\bullet_{\lam,\cmu,\ka}(\g)$ by the formula
$$
\wt{F}^j_{\lam,\ka}(\g) = \bigoplus_{\ell(w)=j} \pi^\ka_{w \circ \lam}.
$$
Let us identify $\pi^\ka_{w \circ \lam - \ka\cmu} \cong \pi^\ka_{w
  \circ \lam}$ as free modules with one generator over the negative
part of the Heisenberg Lie algebra. Then we identify
$\wt{F}^\bullet_{\lam,\ka}(\g)$ and $F^\bullet_{\lam,\cmu,\ka}(\g)$ as
vector spaces. The differential on $F^\bullet_{\lam,\cmu,\ka}(\g)$,
given by formula \eqref{qdiff1}, gives rise to the following
differential on $\wt{F}^\bullet_{\lam,\cmu,\ka}(\g)$. Note that the
screening current $S^W_i(z)$ acting on $\pi^\ka_{\nu-\ka\cmu}$
becomes, under the isomorphism $\pi^\ka_{\nu-\ka\cmu} \cong
\pi^\ka_\nu$ the operator $z^{-\cmu_i}S^W_i(z)$, where as before
$\cmu_i = \bra \cmu,\al_i\rangle$. Thus, the differential
$$
\wt{d}^j_{\lam,\cmu}: \wt{F}^j_{\lam,\cmu,\ka}(\g) \to
\wt{F}^{j+1}_{\lam,\cmu,\ka}(\g)
$$
is given by the same formula \eqref{qdiff1} in which we replace each
$S^W_i(z)$ by $z^{-\cmu_i} S^W_i(z)$. For instance, the $0$th
differential
\begin{equation}    \label{zeroth diff1}
\wt{d}^0_{\lam,\cmu}: \pi^\ka_\lam \to \bigoplus_{i=1}^r
\pi^\ka_{\lam-(\lam_i+1)\al_i}
\end{equation}
is equal to
\begin{equation}    \label{sum scr1}
\wt{d}^0_{\lam,\cmu} = \sum_{i=1}^{r} v_i S^W_{i,(\cmu_i)}(\lam_i+1),
\end{equation}
where
\begin{multline}
S^W_{i,(m)}(n) = \int_{\Gamma} S^W_i(z_1)S^W_i(z_2)\dots S^W_i(z_n)
 z_1^{-m} \dots z_n^{-m} dz_1 \dots dz_n: \\
 \pi^\ka_{\nu}\ra \pi^\ka_{\nu-n\alpha_i}
\label{eq:intertwiner4}
\end{multline}
(compare with formula \eqref{eq:intertwiner3}).

Let us now rescale the generators of the Heisenberg Lie algebra as
follows:
$$
b^i_n \mapsto x^i_n = \frac{b^i_n}{\ka}.
$$
The OPEs \eqref{eq:b_i} imply the commutation relations
$$
[x^i_n,x^j_m] = \frac{1}{\ka} (\al_i|\al_j) n \delta_{n,-m}.
$$
We will consider the Heisenberg algebra and its modules with respect
to these new generators $x^i_n, n \in \Z, i=1,\dots,r$. Then in the
limit $\ka \to \infty$ the Heisenberg algebra becomes commutative,
with generators $x^i_n$. Let us fix once and for all the highest
weight vector in in the Fock representation $\pi^\ka_\nu, \nu \in
\h^*$. Then we can identify $\pi^\ka_\nu$ with $\C[x^i_n]_{n<0}$ (this
corresponds to choosing a particular $\C[\ka^{-1}]$-lattice in
$\pi^\ka_\nu \otimes \C[\ka,\ka^{-1}]$; namely, the one generated by
monomials in the $x^i_n$ applied to the highest weight vector). In the
limit $\ka \to \infty$, we obtain a module on which the negative
subalgebra $\C[x^i_n]$ acts freely and all other generators $x^i_n, n
\geq 0$ act by $0$. Thus, the $\ka \to \infty$ limit of $\pi^\ka_\nu$
defined in this way does not depend on $\nu$. We will denote it simply
by $\pi^\infty$.

According to Lemma 4.3.4 of \cite{FF93}, the screening operator
$Q_i^\ka = \int S^W_i(z) dz: \pi^\ka_0 \to \pi^\ka_{-\al_i}$ has the
following expansion in $\ka^{-1} = \beta^2$:
$$
Q_i^\ka = \ka^{-1} Q_i + \ka^{-2}(\dots),
$$
where bracketed dots represent a power series in non-negative powers
in $\ka^{-1}$ (the difference in sign is due to our choice of sign in
the definition of the screening currents; see Remark on notation in
Section \ref{recoll}). The leading term $Q_i$ is given by formula
(2.2.4) of \cite{FF93} (note that $Q_i = T_i^{-1} \wt{Q}_i$):
\begin{equation}    \label{Qi}
Q_i = \sum_{n<0} S^i_{n+1} \pa^{(i)}_n,
\end{equation}
where the $S^i_n$ are the Schur polynomials given by the generating
function
\begin{equation}    \label{Sin}
\sum_{n\leq 0} S^i_n z^n = \exp(\sum_{m<0} -\frac{x^i_m}{m} z^m)
\end{equation}
and
\begin{equation}    \label{pai}
\pa^{(i)}_n = \sum_{j=1}^r (\al_i|\al_j) \frac{\pa}{\pa
x^j_n}
\end{equation}

In the same way, we obtain an analogous formula for
$$
Q_{i,(\cmu_i)}^\ka = \int S^W_i(z) z^{-\cmu_i} dz: \pi^\ka_0 \to
\pi^\ka_{-\al_i}, \qquad \cmu_i\geq 0.
$$
Namely,
$$
Q_{i,(\cmu_i)}^\ka = \ka^{-1} Q_{i,(\cmu_i)} + \ka^{-2}(\dots),
$$
where
\begin{equation}    \label{Qim}
Q_{i,(\cmu_i)} = \sum_{n<-\cmu_i} S^i_{n+\cmu_i+1} \pa^{(i)}_n,
\end{equation}
Thus, for each $\cmu \in \check{P}_+$ we obtain an $r$-tuple of
operators $Q_{i,(\cmu_i)}$ on the space $\pi^\infty$. These are
actually derivatives of the ring $\pi^\infty = \C[x^i_n]_{n<0}$.

\begin{Lem}
The operators $Q_{i,(\cmu_i)}$ satisfy the Serre relations of $\n_-
\subset \g$:
\begin{equation}    \label{serre}
(\on{ad} f_i)^{-a_{ij}+1} \cdot f_j = 0.
\end{equation}
\end{Lem}

\begin{proof} The proof is essentially the same as the proof of
  Proposition 2.2.8 of \cite{FF93}, which corresponds to the case
  $\cmu=0$. The crucial formula in that proof is the commutation
  relation
\begin{multline}   \label{commrel0}
(\on{ad} Q_i)^m \cdot Q_j = \\ C_m (-a_{ij}-m+1)
\sum_{n_1,\dots,n_{m+1}<0} S^i_{n_1+1} \dots S^i_{n_m+1} S^j_{n_{m+1}+1}
\frac{1}{n_1 \dots n_m} \\ \cdot \left( \sum_{l=1}^m
\frac{n_l}{n_1 + \dots \widehat{n_l} \ldots + n_{m+1}}
\pa^{(i)}_{n_1 + \dots + n_{m+1}} - \pa^{(j)}_{n_1 + \dots +
n_{m+1}} \right),
\end{multline}
where $C_m$ is a constant (note that there is a typo in this formula
in \cite{FF93}; namely, $S^i_{n_1} \dots S^i_{n_m} S^j_{n_{m+1}}$
should be replaced with $S^i_{n_1+1} \dots S^i_{n_m+1}
S^j_{n_{m+1}+1}$). This formula is proved by induction, using the
relations
$$[\pa^{(i)}_n,S^j_m] = - (\al_i|\al_j) \frac{1}{n} S^j_{m-n}$$ (where
we set $S^j_m = 0$, if $m>0$) and the identity
$$\frac{1}{a(a+b)} + \frac{1}{b(a+b)} =
\frac{1}{ab}.$$
The following formula is a straightforward generalization of formula
\eqref{commrel0}:
\begin{multline}   \label{commrelgen}
(\on{ad} Q_i)^m \cdot Q_j = \\ C_m (-a_{ij}-m+1)
\sum_{n_1,\dots,n_m <-\cmu_i; n_{m+1}<-\cmu_j} S^i_{n_1+\cmu_i+1}
\dots S^i_{n_m+\cmu_i+1} S^j_{n_{m+1}+\cmu_j+1}
\frac{1}{n_1 \dots n_m} \\ \cdot \left( \sum_{l=1}^m
\frac{n_l}{n_1 + \dots \widehat{n_l} \ldots + n_{m+1}}
\pa^{(i)}_{n_1 + \dots + n_{m+1}} - \pa^{(j)}_{n_1 + \dots +
n_{m+1}} \right).
\end{multline}
This proves our Lemma.
\end{proof}

According to Proposition 2.4.6 of \cite{FF93}, in the case of $\cmu=0$
the action of $\n_-$ generated by the operators $Q_i, i=1,\dots,r$, on
$\pi^\infty$ is ``cofree'', i.e. $\pi^\infty \cong U(\n_-)^\vee
\otimes V$ for some graded vector space $V$ with a trivial action of
$\n_-$. Here $U(\n_-)^\vee$ is the restricted dual of the free
$\n_-$-module $U(\n_-)$: $U(\n_-)^\vee = \oplus_\gamma
U(\n_-)^*_\gamma$, where for each element $\gamma$ in the root lattice
of $\g$, $U(\n_-)_\gamma$ stands for the corresponding component in
$U(\n_-)$, which is finite-dimensional. In the same way, one can show
that the action of $\n_-$ generated by $Q_{i,(\cmu_i)}, i=1,\dots,r$,
on $\pi^\infty$ is cofree for all $\cmu \in \check{P}_+$ as well.

Now we are ready to study the limit of the complex
$\wt{F}^\bullet_{\lam,\cmu,\ka}(\g)$ as $\ka \to \infty$. We identify
each Fock representation appearing in it with $\C[x^i_n]$ as above,
and in the formula for the differential rescale the screening
current $S^W_i(z) \mapsto \ka S^W_i(z)$. As explained in Sect. 4.6
of \cite{FF93}, the complex defined this way has a well-defined limit
as $\ka \to \infty$.

Let's first look at the limiting complex
$\wt{F}^\bullet_{\lam,\cmu,\infty}(\g)$ in the case $\lam=0, \cmu=0$
considered in \cite{FF93}. It is shown in the proof of Proposition
4.3.5 of \cite{FF93} that the complex
$\wt{F}^\bullet_{0,0,\infty}(\g)$ computes the cohomology of the
complex $\on{Hom}_{\n_-}(B_\bullet(\g),\pi^\infty)$, where
$B_\bullet(\g)$ is the BGG resolution of the trivial representation
$L_0$ of $\g$ (this resolution is the $q \to 1$ limit of the
resolution $B^{q,0}_\bullet(\g)$ discussed in Section \ref{recoll}
above). Since $\pi^\infty$ is a cofree $\n_-$-module, we find that the
$0$th cohomology is $\on{Hom}_{\n_-}(L_0,\pi^\infty) =
(\pi^\infty)^{\n_-}$ and all higher cohomologies vanish.

In the same way, we show that for general $\lam \in P_+, \cmu \in
\check{P}_+$ we have
$$
\wt{F}^\bullet_{\lam,\cmu,\infty}(\g) \simeq
\on{Hom}_{\n_-}(B^\la_\bullet(\g),\pi^\infty)
$$
where $B^\la_\bullet(\g)$ is the BGG resolution of the irreducible
finite-dimensional representation $L_\lam$ of $\g$ (the $q \to 1$
limit of the resolution $B^{q,\lam}_\bullet(\g)$ discussed in Section
\ref{recoll}) and we consider the action of $\n_-$ on $\pi^\infty$
generated by the operators $Q_{i,(\cmu_i)}, i=1,\dots,r$. Since
$\pi^\infty$ is cofree with respect to this action, we obtain

\begin{Pro}
The $0$th cohomology of the complex
$\wt{F}^\bullet_{\lam,\cmu,\infty}(\g)$ is isomorphic to 
$\on{Hom}_{\n_-}(L_\lam,\pi^\infty)$ and all higher cohomologies vanish.
\end{Pro}

\begin{Co}
For generic $\ka$, all higher cohomologies of the complex
$F^\bullet_{\lam,\cmu,\infty}(\g)$ vanish.
\end{Co}

Note that in Theorem \ref{Th:res},(2) we have proved (by relying on
Theorem \ref{Th:van}) a slightly stronger statement: All higher
cohomologies of the complex $F^\bullet_{\lam,\cmu,\infty}(\g)$ vanish
for irrational $\ka$.

\end{document}